\documentclass[english]{amsart}

\usepackage{babel}
\usepackage{amsmath}
\usepackage{amsthm}
\usepackage{amssymb}
\usepackage[utf8]{inputenc}
\usepackage{graphicx}
\usepackage{float}
\usepackage[colorlinks = true, urlcolor = blue, linkcolor = blue, citecolor = blue, breaklinks = true]{hyperref}
\usepackage{breakcites}
\usepackage[dvipsnames]{xcolor}
\usepackage{tikz-cd}
\usetikzlibrary{spath3, intersections}
\usetikzlibrary{
  knots,
  hobby,
  decorations.pathreplacing,
  shapes.geometric,
  calc
}

\tikzset{
  knot diagram/every strand/.append style={
    ultra thick,
    red
  },
  show curve controls/.style={
    postaction=decorate,
    decoration={show path construction,
      curveto code={
        \draw [blue, dashed]
        (\tikzinputsegmentfirst) -- (\tikzinputsegmentsupporta)
        node [at end, draw, solid, red, inner sep=2pt]{};
        \draw [blue, dashed]
        (\tikzinputsegmentsupportb) -- (\tikzinputsegmentlast)
        node [at start, draw, solid, red, inner sep=2pt]{}
        node [at end, fill, blue, ellipse, inner sep=2pt]{}
        ;
      }
    }
  },
  show curve endpoints/.style={
    postaction=decorate,
    decoration={show path construction,
      curveto code={
        \node [fill, blue, ellipse, inner sep=2pt] at (\tikzinputsegmentlast) {}
        ;
      }
    }
  }
}
\usepackage{url}
\usepackage{dcolumn}
\usepackage{xcolor}
\usepackage{tikz}
\usetikzlibrary{decorations.markings}
\usepackage{mathtools}
\usepackage{nicematrix}
\usepackage{mathrsfs}
\usepackage{booktabs}
\usepackage{pdfpages}
\usepackage{bbm}
\usepackage{enumitem}
\usepackage{scalerel}
\usetikzlibrary{matrix,arrows,chains}
\newtheorem*{theorem*}{Theorem}
\newtheorem{prop}{Proposition}[section]
\newtheorem{thm}[prop]{Theorem}
\newtheorem*{thm*}{Theorem}
\newtheorem{cor}[prop]{Corollary}
\newtheorem{lem}[prop]{Lemma}

\theoremstyle{definition}
\newtheorem{rem}[prop]{Remark}
\newtheorem{de}[prop]{Definition}
\newtheorem{ex}[prop]{Example}

\newtheorem{con}[prop]{Construction}
\newtheorem{quest}[prop]{Question}
\newtheorem*{quest*}{Question}
\newtheorem*{cor*}{Corollary}

\usepackage[left=3cm,right=3cm]{geometry}
\usepackage{stackengine,wasysym}

\newcommand\reallywidetilde[1]{\ThisStyle{%
  \setbox0=\hbox{$\SavedStyle#1$}%
  \stackengine{-.1\LMpt}{$\SavedStyle#1$}{%
    \stretchto{\scaleto{\SavedStyle\mkern.2mu\AC}{.5150\wd0}}{.4\ht0}%
  }{O}{c}{F}{T}{S}%
}}

\makeatletter
\tikzcdset{
  open/.code     = {\tikzcdset{hook, circled};},
  closed/.code   = {\tikzcdset{hook, slashed};},
  open'/.code    = {\tikzcdset{hook', circled};},
  closed'/.code  = {\tikzcdset{hook', slashed};},
  circled/.code  = {\tikzcdset{markwith = {\draw (0,0) circle (.375ex);}};},
  slashed/.code  = {\tikzcdset{markwith = {\draw[-] (-.0ex,-.4ex) -- (.0ex,.4ex);}};},
  markwith/.code ={
    \pgfutil@ifundefined%
    {tikz@library@decorations.markings@loaded}%
    {\pgfutil@packageerror{tikz-cd}{You need to say %
      \string\usetikzlibrary{decorations.markings} to use arrows with markings}{}}{}%
    \pgfkeysalso{/tikz/postaction = {
      /tikz/decorate,
      /tikz/decoration={markings, mark = at position 0.5 with {#1}}}
    }
  },
}
\makeatother

\DeclareMathOperator{\N}{\mathbb{N}}
\DeclareMathOperator{\Z}{\mathbb{Z}}

\DeclareMathOperator{\R}{\mathbb{R}}
\DeclareMathOperator{\C}{\mathbb{C}}
\DeclareMathOperator{\A}{\mathbb{A}}
\DeclareMathOperator{\PP}{\mathbb{P}}
\renewcommand{\P}{\PP}
\DeclareMathOperator{\G}{\mathbb{G}}

\DeclareMathOperator{\id}{id}

\DeclareMathOperator{\pr}{pr}

\DeclareMathOperator{\wo}{\hspace{-1pt}\setminus}

\DeclareMathOperator{\komma}{\hspace{-3pt},}

\DeclareMathOperator{\ab}{ab}

\DeclareMathOperator{\Fun}{Fun}

\DeclareMathOperator{\Hom}{Hom}
\DeclareMathOperator{\End}{End}
\DeclareMathOperator{\Aut}{Aut}

\DeclareMathOperator{\op}{op}

\renewcommand{\ker}{\kerr}
\DeclareMathOperator{\kerr}{ker}

\renewcommand{\lim}{\limi}
\DeclareMathOperator{\limi}{lim}

\DeclareMathOperator{\Ab}{Ab}

\DeclareMathOperator{\PGL}{PGL}
\DeclareMathOperator{\SL}{SL}

\DeclareMathOperator{\alg}{alg}

\DeclareMathOperator{\Spc}{Spc}

\DeclareMathOperator{\SH}{SH}

\DeclareMathOperator{\Map}{Map}
\DeclareMathOperator{\Deck}{Deck}

\DeclareMathOperator{\Sm}{Sm}

\DeclareMathOperator{\Spec}{Spec}

\DeclareMathOperator{\Nis}{Nis}

\DeclareMathOperator{\Gm}{\mathbb{G}_m}

\DeclareMathOperator{\Shv}{Shv}
\DeclareMathOperator{\PShv}{PShv}

\DeclareMathOperator{\K}{K}
\DeclareMathOperator{\KK}{\underline{K}}

\DeclareMathOperator{\GW}{GW}

\DeclareMathOperator{\MM}{M}
\DeclareMathOperator{\QQ}{Q}

\DeclareMathOperator{\MW}{MW}

\DeclareMathOperator{\cofib}{cofib}

\DeclareMathOperator{\Th}{Th}

\graphicspath{ {./Bilder/} }
\tikzset{double line with arrow/.style args={#1,#2}{decorate,decoration={markings,%
mark=at position 0 with {\coordinate (ta-base-1) at (0,1pt);
\coordinate (ta-base-2) at (0,-1pt);},
mark=at position 1 with {\draw[#1] (ta-base-1) -- (0,1pt);
\draw[#2] (ta-base-2) -- (0,-1pt);
}}}}

\tikzset{%
    symbol/.style={%
        draw=none,
        every to/.append style={%
            edge node={node [sloped, allow upside down, auto=false]{$#1$}}}
    }
}

\tikzset{
  open/.style = {decoration = {markings, mark = at position 0.5 with { \node[transform shape, scale = .8] {$\circ$}; } }, postaction = {decorate} }
}

\begin{document}

\title{Algebraic Knots and their Universal \texorpdfstring{$\KK^{\MW}_2$}{KMW2}-Coverings}

\author[Wendt]{Matthias Wendt}
\address[Wendt]{
Fachgruppe Mathematik/Informatik, Bergische Universit\"{a}t Wuppertal, 
Gaußstraße 20, 42119 Wuppertal,
Germany}
\email{mwendt@uni-wuppertal.de}

\author[Wittich]{Thor Wittich}
\address[Wittich]{
Department of Mathematics, Universit\"{a}t Osnabr\"{u}ck, 
Albrechtstrasse 28a, 49076 Osnabr\"{u}ck,
Germany}
\email{thor.wittich@uni-osnabrueck.de}

\begin{abstract}
Over suitable base fields $k$ of characteristic not $2$, including algebraically closed ones, we construct universal abelian $\KK^{\MW}_2$-coverings for complements of closed embeddings $\A^1 \hookrightarrow \A^3$. Using these, we obtain a rectifiability invariant of such embeddings by lifting knot-theoretic ideas to algebraic geometry via motivic homotopy theory.
\end{abstract}
\maketitle

{ \hypersetup{hidelinks} \tableofcontents }

\section{Introduction}
Algebraic geometry deals with the study of geometric objects locally defined by polynomial equations. Independent of the setup or the choice of language, affine spaces are among the simplest and most central examples in (affine) algebraic geometry. Despite that, there are still quite some open problems about properties of affine spaces. One of the most famous of such problems, if not the most famous one, is the Zariski cancellation problem. It is often phrased for algebraically closed base fields, which is the most interesting case.
\begin{quest*}[Zariski cancellation problem]\label{ZCPIntro}
Let $k$ be a base field. Is there an affine scheme $X$ such that $X \times \A^1 \cong \A^{n+1}$, but $X \not\cong \A^n$?
\end{quest*}
In other words, the question is whether taking products with $\A^1$ is cancellative. As we will discuss in more detail in the next section, this problem has been solved completely for fields of positive characteristic, but remains wide open in characteristic zero. The former is mainly due to the seminal articles \cite{MR3148104} and \cite{MR3250286} of Gupta, in which for all $n \geq 3$, the existence of affine schemes $X$ as in the above question is proven. For the case  $n = 2$, the answer is negative for all base fields. This is due to Fujita, Miyanishi and Sugie in \cite{MR531454} and \cite{MR564667} in characteristic zero, and due to Russell \cite{MR615851} for perfect fields. Finally, Bhatwadekar and Gupta \cite{MR3368259} extended Russell's result to all fields.

\hfill

A related problem is Abhyankar's rectification problem \cite{AbMoh}, sometimes also called embedding problem or epimorphism problem, which asks about the existence of non-rectifiable embeddings between affine spaces. This is, once again, typically phrased over an algebraically closed base field, but can be stated more generally.

\begin{quest*}[Rectification problem]\label{RectifiabilityIntro}
Let $k$ be a base field. Given a closed embedding $i \colon \A^m \hookrightarrow \A^n$, does there exist an automorphism $\varphi \in \Aut(\A^n)$ such that $\varphi \circ i$ becomes a standard embedding?
\end{quest*}
Here a standard embedding is one given by inclusion as coordinate axes. If there is such a $\varphi$ for a given closed embedding $i \colon \A^m \hookrightarrow \A^n$, the embedding $i$ is called rectifiable, and if not, it is called non-rectifiable. The relation to the Zariski cancellation problem was established by Asanuma \cite{MR1720185}. There it is proven that any non-rectifiable embedding gives rise to a potential counterexample for Zariski cancellation, i.e., from any non-rectifiable embedding $i \colon \A^m \hookrightarrow \A^n$ we can obtain an affine scheme $X$ (of dimension $n+1$), which satisfies $X \times \A^m \cong \A^{\dim(X) + m}$, and which is complicated enough to have a chance at not being an affine space itself. This is also how Gupta's counterexamples to Zariski cancellation in positive characteristic were produced in the aforementioned seminal articles \cite{MR3148104} and \cite{MR3250286}. For more details, we refer to the discussion around Theorem \ref{Asanuma}.

As explained just now, the existence of non-rectifiable embeddings is an interesting question, but there is a major problem. There are no computable invariants for detecting rectifiability. In a series of (at least) two articles, to which this one belongs, we remedy this. The main result of this article lays the foundation of our ``motivic knot theory'' and states the following: 
\begin{thm*}[Construction \ref{Construction of Milnor--Witt K_2 coverings}, Lemma \ref{Unknot is Univ Covering}, Theorem \ref{Coverings are Universal}, Lemma \ref{Alexander mods detect alg knot types} and Corollary \ref{Rectifiability invariant}]
Let $k$ be an algebraically closed base field of characteristic not $2$. The complement $X = \A^3 \setminus i(\A^1)$ of any closed embedding $i \colon \A^1 \hookrightarrow \A^3$ admits a universal abelian $\A^1$-covering $p \colon \widetilde{X} \rightarrow X$ with the following properties:
\begin{enumerate}
\item[(i)] The fibers of $p$ are given by the discrete motivic space $\KK^{\MW}_2$.
\item[(ii)] If the embedding $i$ is given by inclusion of coordinate axes, then $p \colon \widetilde{X} \rightarrow X$ is universal.
\item[(iii)] If there is an automorphism $\varphi \in \Aut(\A^3)$ which maps the embedding $i$ to an embedding $i'$, then the $\A^1$-homology sheaves $\underline{H}_1^{\A^1}\!(\widetilde{X})$ and  $\underline{H}_1^{\A^1}(\widetilde{X'})$ agree, where $p' \colon \widetilde{X'} \rightarrow X'$ is the universal abelian $\A^1$-covering associated with $i'$.
\end{enumerate}
In particular, if $\underline{H}_1^{\A^1}\!(\widetilde{X})$ is nontrivial, then $i$ is non-rectifiable.
\end{thm*}
The idea behind this is to treat embeddings $\A^1 \hookrightarrow \A^3$ like knots. Under this perspective, being non-rectifiable is the same as being a non-trivial algebraic knot in the sense of the rectification problem, so that we need to construct invariants of algebraic knots to detect rectifiability. For this we make use of motivic homotopy theory, which enables $\A^1$-algebraic topological tools for us. In particular, the sheaf $\underline{H}_1^{\A^1}\!(\widetilde{X})$ should be viewed as a motivic version of the classical knot invariant known as Alexander module. The basic idea behind this is not completely novel. Over $k = \R$, Shastri \cite{MR1145717} used knot theoretic ideas to show that the trefoil knot gives rise to a non-rectifiable embedding $\A^1 \hookrightarrow \A^3$. In some sense, this paper can be seen as an extension of Shastri's idea to more general base fields, where usual knot--theoretic notions are not meaningful, or even definable.

As will become apparent from Construction \ref{Construction of Milnor--Witt K_2 coverings}, explicit computations require us to be able to write a given embedding $\A^1 \hookrightarrow \A^3$ as a complete intersection, which is always possible by Theorem \ref{Knots are Complete Intersections}. Achieving this in practice is a highly non-trivial matter, as we will discuss just before Question \ref{Rectifiability}. By now we are able to write the so-called algebraic trefoil knot (see Example \ref{Knot picture}) as a complete intersection and are working on a follow-up article \cite{WW2} in which we want to use this for explicit computations.

\hfill

The proof of the above main result makes crucial use of the Hurewicz map $h\colon\underline{\pi}_1^{\A^1}\!(X)\to\underline{H}_1^{\A^1}\!(X)$ being an epimorphism for algebraic knot complements $X$. In \cite{MR4442407}, Choudhury and Hogadi gave a proof of this statement for any space $X$ over $k$, but unfortunately their argument is known to be incomplete; see, for instance, \cite{zbMATH08029975}. The main issue for proving this is that it is not known that quotients of strongly $\A^1$-invariant sheaves of groups remain strongly $\A^1$-invariant. 

To remedy this, we prove in Appendix \ref{Appendix B} that the Hurewicz map for knot complements is indeed an epimorphism  under suitable assumptions.
\begin{thm*}[Theorem \ref{Hurewicz}, Theorem \ref{thm:hurewicz-knot-complement}, Remark \ref{rem:assumptions for hurewicz}]
Let $k$ be an algebraically closed base field of characteristic not $2$ (or $k = \R$) and let $i \colon \A^1 \hookrightarrow \A^3$ be a closed embedding (defined by Chebyshev polynomials) with complement $X = \A^3 \setminus i(\A^1)$. Then the Hurewicz map $h\colon\underline{\pi}_1^{\A^1}\!(X)\to\underline{H}_1^{\A^1}\!(X)$ is an epimorphism.
\end{thm*}
Our main reason for pushing the proof of this result into the appendix is that we felt it distracts the reader from the other main result and its proof in Section \ref{Section 4}. As essentially no non-trivial case of the Hurewicz theorem is known, we consider this result interesting by itself.

\hfill

This is not the first variant of ``motivic knot theory''. In the article \cite{Clementine}, a version of motivic linking numbers are considered. These two “motivic knot theories” both “motivify” certain aspects of classical knot theory, but not the general theory. 

\subsection*{Acknowledgement} The authors would like to thank Aravind Asok for thinking about the first version of these ideas together with the first-named author at the Mittag-Leffler institute many years ago. Furthermore, TW wishes to thank Marco Giustetto and Jesse Pajwani for useful discussions. 

\section{Algebraic Knots}
In topology, a knot is an embedding $S^1 \hookrightarrow S^3$ and a knot type is an isotopy class of such embeddings. Given a knot 
$$i \colon \R^1 \cup \, \lbrace \infty \rbrace \cong S^1 \hookrightarrow S^3 \cong \R^3 \cup \, \lbrace \infty \rbrace,$$
where we consider both of the spheres as one-point compactifications of Euclidean spaces, we can always choose the point at infinity on the right hand side to be the image of the one on the left hand side. This allows us to restrict away from the points at infinity, so that we obtain an embedding $i \colon \R \hookrightarrow \R^3$. Note that, as this embedding is obtained from an embedding between compact Hausdorff spaces, it is proper, i.e., it has the property that preimages of compact subsets are again compact. Therefore it extends to a map on one-point compactifications, so that we can get our original knot $i \colon S^1 \hookrightarrow S^3$ back. This is the reason why proper embeddings $i \colon \R \hookrightarrow \R^3$ are also called proper knots. To summarize, we have a one-to-one correspondence
\begin{center}
\begin{tikzcd}
\lbrace \text{Knots } S^1 \hookrightarrow S^3 \rbrace \arrow[rr, shift left=5pt, "\text{decompactify}", "1:1"'] & & \lbrace \text{Proper knots } \R^1 \hookrightarrow \R^3 \rbrace. \arrow[ll, shift left=5pt, "\text{compactify}"]
\end{tikzcd}
\end{center}
There is also a suitable notion of isotopy classes for proper knots, which we do not want to recall, but which allows us to extend this correspondence to knot types. For more on this, see Section 2 of \cite{MR1145717}.
\begin{rem}
If one just takes the usual notion of isotopy, there is only one proper knot type. This is often done in the literature, e.g., when proper knots in other manifolds than $\R^3$ are studied. We want to highlight this to avoid confusion.
\end{rem}
\begin{rem}
Note that on the level of knot complements, there is no difference between a knot and its associated proper knot. So without any additional work, we already know quite some invariants of proper knots.
\end{rem}

The reason that we consider proper knots, as opposed to the usual notion of knots, is the following result of Shastri \cite{MR1145717}.
\begin{thm}
Every proper knot type is algebraic, that is, every proper knot is equivalent to one given by polynomial maps.
\end{thm}
As it turns out, Chebyshev polynomials of the first kind give rise to many such knots, as was shown in \cite{freudenburg:freudenburg}. These are relevant examples for us, which is why we will address those more in Appendix \ref{Appendix A}. For now we will consider a few explicit examples, which are taken from \cite{MR1145717}:

\begin{ex}\label{Knot picture}
The unknot, trefoil knot and figure-eight knot
\begin{center}
	\centerline{\includegraphics[scale=1]{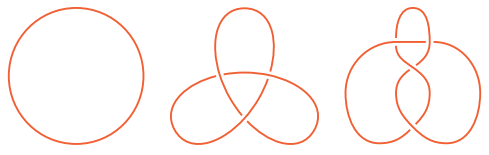}}    
\end{center}
can be represented by the three polynomial maps
$$t \mapsto (t,0,0), \,\,\, t \mapsto (t^3 - 3t, t^4 - 4t^2, t^5 - 10t), \,\,\,\text{and} \,\,\,t \mapsto (t^3 - 3t, t(t^2-1)(t^2-4), t^7 - 42t),$$
respectively.
\end{ex}

In light of the above theorem, we are able to study knots within the world of algebraic geometry. As maps between affine schemes are always quasi-compact (proper in the topological terminology), the basic object of study will just be the following.
\begin{de}
An algebraic knot is a closed embedding $i \colon \A^1 \hookrightarrow \A^3$, that is, a closed immersion with $i(\A^1) \cong \A^1$.
\end{de}
Embeddings into affine spaces have, of course, been studied extensively. For instance, the following is known by a culmination of results of Ferrand--Szpiro from \cite{MR379517} and \cite{MR572085}, Boraty\'{n}ski \cite{MR511453}, Kumar \cite{MR472851} and Quillen--Suslin from \cite{MR427303} and \cite{MR469905}.
\begin{thm}
Let $k$ be a base field and let $i \colon X \hookrightarrow \A^n$ be an embedding of an $m$-dimensional affine scheme such that the normal bundle $\mathcal{N}_{X/\!\A^n}$ is trivial. If $i(X)$ is a local complete intersection, then
\begin{enumerate}\label{Knots are Complete Intersections}
\item[(i)] $i(X)$ is a set-theoretic complete intersection.
\item[(ii)] $i(X)$ is a complete intersection, if $n-m = 2$, or if $n \geq 2(m+1)$.
\end{enumerate}
In particular, algebraic knots are complete intersections.
\end{thm}
\begin{proof}
Part (i) was proven for $n = 3$ by Ferrand--Szpiro in \cite{MR379517} and \cite{MR572085}, and then generalized by Boraty\'{n}ski \cite{MR511453}. The first part of (ii) follows from Boraty\'{n}ski's proof of (i), as also observed by Forster in \cite{MR775875}, and the second part of (ii) is due to Kumar \cite{MR472851}. 

For the in particular part, let $i \colon \A^1 \hookrightarrow  \A^3$ be an algebraic knot. By the Quillen--Suslin theorem (\cite{MR427303} and \cite{MR469905}), all vector bundles on affine spaces are trivial, so that $\mathcal{N}_{\A^1\!/\A^3}$ is trivial. As the codimension is $2$, it suffices to show that $i(\A^1) \cong \A^1$ is a local complete intersection by part (ii). But this is clear as $\A^1$ is regular and hence a local complete intersection.
\end{proof} 
For those readers with little algebro-geometric background, let us clarify what this result gives us. The statement that an algebraic knot is a complete intersection means that $i(\A^1)$ can be described as the common roots of two polynomials, i.e., there exist $f,g \in k[x,y,z]$ with 
$$i(\A^1) = V(f) \cap V(g) = V(f,g).$$
This will become crucial for constructing certain covering maps and comparison maps later on; see, for instance, Construction \ref{Construction of Milnor--Witt K_2 coverings}. 

For a concrete knot, such as the non-trivial ones from Example \ref{Knot picture}, it is truly difficult to find such two polynomial generators. It seems that the usual computer algebra tools and packages of Macaulay2, Singular, etc. are not optimized for this task and only manage to produce generating sets of size 3 at best. Despite this, we have by now managed to find explicit generators describing the algebraic trefoil knot as a complete intersection and will use these for computations in our ongoing follow-up article \cite{WW2}.

\hfill

One central question of Abhyankar \cite{AbMoh} from affine algebraic geometry concerns the so-called rectifiability of closed embedding $i \colon \A^m \hookrightarrow \A^n$. To state the question, let us call such an embedding standard, if it is given by including $\A^m$ as coordinate axes.
\begin{quest}\label{Rectifiability}
Let $k$ be a base field of characteristic $p \geq 0$. Given a closed embedding $i \colon \A^m \hookrightarrow \A^n$, does there exist an automorphism $\varphi \in \Aut(\A^n)$ such that $\varphi \circ i$ becomes a standard embedding?
\end{quest}

If so, the embedding $i \colon \A^m \hookrightarrow \A^n$ is called rectifiable. While the main interest in this question concerns the case $k = k^{\alg}$, let us include a few further examples in the following list of prominent cases:
\begin{enumerate}
\item[$\bullet$] If $m = 1$ and $n = 2$, there is a full answer because the automorphism group $\Aut(\A^2)$ is understood well enough:
\begin{enumerate}
\item[$\bullet$] If $p = 0$, then the answer is yes. This is a result due to Abhyankar and Moh \cite{MR379502} for $k = \C$, and Suzuki \cite{MR338423} for the generalization to arbitrary fields of characteristic $0$.
\item[$\bullet$] If $p > 0$, then the answer is no, in general. Already Segre and Nagata provided examples of such embeddings in \cite{MR95175} and \cite{MR337962}, respectively. To give a few more details, Nagata's example is cut out by the polynomial $f = x^{p^2} + y + y^{qp}$, where $q$ can be any prime other than $p$, and can be found on page 39 of loc.\ cit.
\end{enumerate}
\item[$\bullet$] If $m = 1$ and $n = 3$, Shastri showed in \cite{MR1145717} that the algebraic trefoil knot from Example \ref{Knot picture} is non-rectifiable over $\R$, and conjectured that the same is true over $\C$. 
\item[$\bullet$] If $k = k^{\alg}$ and $n \geq 2(m+1)$, the answer is yes independently of $p$. This can be found in Srinivas' article \cite{MR1087241}, but is, also there, attributed to Nori. In particular, if $m = 1$, then the answer is yes for all $n \geq 4$.

\item[$\bullet$] If $n = m+1$, then this is known as Abhyankar--Sathaye's embedding problem. The only understood case is the one from above, but the expected answer is no even in characteristic $0$.
\end{enumerate}

Let us focus a bit more on the case $m = 1$. If we were to call an embedding $i \colon \A^1 \hookrightarrow \A^n$ an algebraic knot of codimension $n - 1$, then some of the results mentioned in the table could be rephrased as the following slogan:
\begin{center}
 ``There are no non-trivial algebraic knots over $\C$ outside of codimension $2$.''
\end{center}
This is exactly analogous to the topological situation with the main difference being that it is not difficult at all to show the existence of non-trivial (topological) knots in codimension $2$. 

\hfill

As it turns out, Question \ref{Rectifiability} is related to another central problem of affine algebraic geometry:
\begin{quest}[Zariski cancellation problem]
Let $k$ be a base field of characteristic $p \geq 0$. Is there an affine scheme $X$ such that $X \times \A^1 \cong \A^{n+1}$, but $X \not\cong \A^n$?
\end{quest}
Once again, this question is mostly considered for algebraically closed fields $k$. It is certainly difficult to imagine a scheme as in the above question, which is why the expected answer was no. As it turns out, if $p > 0$, the answer is yes in all dimensions $n \geq 3$, as proven by Gupta in \cite{MR3148104} and \cite{MR3250286}:
\begin{thm}[{\cite[Theorem 3.6]{MR3250286}}]\label{Gupta}
Let $k$ be a base field of characteristic $p > 0$. For all $n \geq 3$ there exists an affine scheme $X$ of dimension $n$ such that $X \times \A^1 \cong \A^{n+1}$ and $X \not \cong \A^n$.
\end{thm}
The big problem is, of course, how to find suitable candidates for such schemes $X$, if they exist. This is where the rectification problem (Question \ref{Rectifiability}) comes into play. In \cite{MR1720185}, Asanuma proved the following result.
\begin{thm}\label{Asanuma}
Let $k$ be a base field and let $i \colon \A^m \hookrightarrow \A^n$ be a non-rectifiable embedding and write $i(\A^m) = V(f_1,\dotsc,f_{n-m})$ for polynomials $f_1,\dotsc,f_{n-m} \in k[x_1,\dotsc,x_n]$. Then 
$$\Spec(k[t,x_1,\dotsc,x_n,y_1,\dotsc,y_{n-m}]/(t^ry_1 - f_1,\dotsc,t^ry_{n-m}-f_{n-m})) \times \A^m \cong \A^{n+m+1}$$
for all integers $r \geq 1$. 
\end{thm}
If we work in positive characteristic, we can apply this to the non-rectifiable embeddings $\A^1 \hookrightarrow \A^2$ of Segre \cite{MR95175} or Nagata \cite{MR337962} that we listed above. As the reader may suspect, these potential counterexamples for Zariski cancellation were indeed the ones that Gupta verified to be actual counterexamples in \cite{MR3148104}. These counterexamples were then modified by Gupta \cite{MR3250286} to yield counterexamples for all $n \geq 3$ as stated in Theorem \ref{Gupta}.

\hfill

This begs the question why this does not work in characteristic $0$. First of all, we are not in the fortunate situation that we have non-rectifiable embeddings $\A^1 \hookrightarrow \A^2$. This forces us to consider $n \geq 3$. Here we can use the algebraic trefoil knot based on the results of Shastri frm \cite{MR1145717}, but so far nobody managed to prove that the resulting Asanuma variety is not an affine space; see also the end of section 2 of \cite{zbMATH07823034}. 

In general, since we do not understand the automorphism groups of those higher-dimensional affine spaces well enough, and have no good rectifiability invariants for closed embeddings $\A^m \hookrightarrow \A^n$, it seems rather difficult to find non-rectifiable embeddings and hence to get potential counterexamples to Zariski cancellation. The main goal of this article is to explain how to obtain an invariant of embeddings $\A^1 \hookrightarrow \A^3$ using motivic homotopy theory, which by definition seems rather computable and hence useful. As mentioned before, this is investigated in more detail in our follow-up article \cite{WW2}.

\section{Recollection of some Unstable Motivic Homotopy Theory}
As this article could be particularly interesting to (affine) algebraic geometers, we give a brief introduction to some aspects of unstable motivic homotopy theory. Although we choose to work in a different language, we refer to \cite{MR3727503} for more details.

\hfill

Certain results that we want to use are not known for general qcqs base schemes $S$, so that we restrict ourselves to a base more suitable for our purposes, namely to a perfect base field $k$ of characteristic not $2$. This is the setup for which most is known in motivic homotopy theory.

\begin{de}
The category $\Spc(k)$ of motivic spaces over $k$ is the full subcategory of the category $\PShv(\Sm_k) = \Fun(\Sm_k^{\op},\Spc)$ generated by those presheaves that satisfy Nisnevich descent and that are $\A^1$-invariant.
\end{de}
We will not explain why the definition is exactly the way it is, but we still want to highlight that one key aspect of this definition is that the affine line $\A^1$ becomes contractible. If the reader has never seen the Nisnevich topology, we recommend not to worry too much about it when it comes to understanding this article. 

\hfill

There are two central kinds of examples of motivic spaces:
\begin{ex}
Every smooth scheme $X$ over $k$ gives rise to a motivic space over $k$ as follows. We can consider its associated presheaf on $\Sm_k$, which is a Nisnevich sheaf as the Nisnevich topology is subcanonical, that is, all representable presheaves are sheaves. If $X$ also happens to be $\A^1$-invariant, we are done. This is the case, for instance, for $\G_{\rm m} = \A^1 \setminus \lbrace 0 \rbrace$ and boils down to the fact that the units of $k[x]$ agree with those of $k$ (or, more generally, units of $R[x]$ agree with units of $R$ for reduced $k$-algebras). If $X$ is not $\A^1$-invariant, we have to enforce $\A^1$-invariance in such a way that we do not lose Nisnevich descent, which can be done; see \cite{MR3727503}. 
\end{ex}
\begin{ex}
Every space $X$ yields a motivic space over $k$. For this we just have to consider the constant Nisnevich sheaf on $\Sm_k$ with value $X$. In particular, the $n$-sphere $S^n$ defines a motivic space over $k$ for any $n \geq 0$.
\end{ex}
As in topology, we also want to consider pointed variants of our objects, i.e., we want to consider pointed motivic spaces $\Spc(k)_*$. This category can either be defined as the slice category $\Spc(k)_{/ \Spec(k)}$ or as the full subcategory of $\PShv(\Sm_k^{\op},\Spc_*)$ generated by those $\Spc_*$-valued presheaves on $\Sm_k$ that satisfy Nisnevich descent and which are $\A^1$-invariant. 
\begin{ex}
Typical examples of pointed motivic spaces are: 
\begin{enumerate}
\item[$\bullet$] $(\mathbb{A}^n,e_1)$ with $e_1 = (1,0,\dotsc,0)$ and in particular $(\mathbb{A}^1,1)$
\item[$\bullet$] $(\mathbb{A}^n \setminus \lbrace 0 \rbrace,e_1)$ and in particular $(\G_{\rm m},1)$
\item[$\bullet$] $(\mathbb{P}^1,\infty)$
\item[$\bullet$] $(S^n,1)$
\end{enumerate}
These occur rather frequently, so that we usually drop the base point from the notation. 
\end{ex}

As spheres clearly play a crucial role in homotopy theory, we need to figure out what ``motivic spheres'' are. To motivate the definition, note that the usual spheres $S^n$ can all be defined in terms of the $1$-dimensional sphere $S^1$. Indeed, using the smash product ``$\wedge$'', a ``tensor product'' of pointed spaces, we have $S^n \simeq (S^1)^{\wedge n}$ for all $n \geq 0$. 

If we, as mentioned above, model pointed motivic spaces as a full subcategory of $\PShv(\Sm_k^{\op},\Spc_*)$, it is not difficult to imagine that we can lift the smash product from $\Spc_*$ to a smash product on $\Spc(k)_*$ satisfying the usual properties, including the fact that its unit is $S^0$; see Section 4.5 of \cite{MR3727503}. Therefore we can try to mimic the topological situation and define motivic spheres as smash products of $1$-dimensional motivic spheres, for which we certainly have two candidates in $\Spc(k)_*$:
$$\text{topological }1\text{-sphere }S^1 \hspace{70pt} \text{algebraic }1\text{-sphere }\G_{\rm m}$$
In analogy with the topological spheres, this leads to the following.
\begin{de}
Let $p$ and $q$ be non-negative integers with $p \geq q$. The motivic sphere $S^{p,q}$ of bidegree $(p,q)$ is the (pointed) motivic space $(S^1)^{\wedge (p-q)} \wedge \mathbb{G}_m^{\wedge q}.$
\end{de}
Before we consider a few examples, we want to comment on the discrepancy between the first component of the bidegree and the number of copies of $S^1$. This convention, which is the standard one in motivic homotopy theory, is chosen in a way that it aligns well with the bidegree of motivic cohomology.
\begin{ex}
By definition we have $S^{1,0} = S^1$ and $S^{1,1} = \G_{\rm m}$.
\end{ex}
\begin{ex}
We have $S^{2,1} \simeq \P^1$. Recall that $\P^1$ can be defined/constructed by gluing two copies of $\A^1$ along $\G_{\rm m}$, i.e., we have a pushout square
\begin{center}
\begin{tikzcd}
\G_{\rm m} \arrow[r] \arrow[d] & \mathbb{A}^1 \arrow[d] \\
\mathbb{A}^1 \arrow[r] & \mathbb{P}^1 \arrow[ul, phantom, "\ulcorner", very near start]
\end{tikzcd}
\end{center}
of smooth schemes over $k$. As it turns out, this is not only a pushout square in $\Sm_k$, but also in $\Spc(k)_*$ if we consider $\P^1$ as a pointed motivic space with base point $1$; see, for instance, Proposition 4.13 of \cite{MR3727503}. But $\A^1$ is contractible, so that this pushout square is equivalent to
\begin{center}
\begin{tikzcd}
\G_{\rm m}  \arrow[r] \arrow[d] & * \arrow[d] \\
* \arrow[r] & (\mathbb{P}^1,1) \arrow[ul, phantom, "\ulcorner", very near start].
\end{tikzcd}
\end{center}
As such pushouts are computed by applying $\Sigma \simeq S^1 \wedge -$ to the object in the upper left corner, we therefore have
$$(\P^1,1) \simeq \Sigma \!\G_{\rm m} \simeq S^1 \wedge \G_{\rm m} = S^{2,1}.$$
Finally, note that the automorphism
$$\left ( \begin{array}{cc}
0 & 1 \\

1 & -1

\end{array} \right ) \in \PGL_2(k) = \Aut(\mathbb{P}^1)$$
maps $1$ to $\infty$, and the claim follows.
\end{ex}
In the argument we used that suspensions $\Sigma$ agree with smashing with $S^1$, which is true for any $\infty$-category which has coproducts, so that suspensions exist, and which is tensored over $\Spc_*$, so that we can smash with $S^1$. For this one just uses the hom-tensor adjunction to verify the universal property of pushouts.
\begin{ex}\label{Punctured Affine Space as Sphere}
Similarly, an induction on $n$ yields $\A^n \setminus \lbrace 0 \rbrace \simeq S^{2n-1,n}$ as shown, for instance, in Proposition 4.40 of \cite{MR3727503}.
\end{ex}
Now that we have spheres, we can define homotopy groups in the usual way, that is, as homotopy classes of maps out of spheres. In motivic homotopy theory, these are usually defined as sheaves.
\begin{de}
The $(p,q)$-th $\A^1$-homotopy sheaf $\underline{\pi}_{p,q}^{\A^1}(X,x)$ of a pointed motivic spaces $(X,x)$ is the Nisnevich sheafification of the presheaf $\widetilde{\phantom{\pi}}\hspace{-6pt}\pi_{p,q}^{\A^1}(X,x)$ on $\Sm_k$ given by mapping $U \mapsto [S^{p,q}\wedge U_+,(X,x)]_*$.
\end{de}
Here the notation $[-,-]_*$ denotes homotopy classes of pointed maps, as per usual. While we can define $\A^1$-homotopy sheaves in bidegree $(p,q)$ for all $p \geq q \geq 0$, there is an algebraic incarnation of $\G_{\rm m}$-loop spaces called (homotopy) contraction; see page 33 of \cite{MR2934577}. Therefore, it is rather common to focus on homotopy sheaves with respect to the topological spheres only, which are denoted by $\underline{\pi}_p^{\A^1}\!(X,x)$ to simplify notation. 

\hfill

In terms of these we can also define what it means to be $\A^1$-discrete. For a motivic space $X$, we can consider the map 
$$\pi_0 \colon X(-) \simeq \Map(-,X) \rightarrow [-,X],$$
which gives us a natural map $X \rightarrow \underline{\pi}_0^{\A^1}\!(X)$ after sheafification.
\begin{de}
A motivic space $X$ is $\A^1$-discrete, if the natural map $X \rightarrow \underline{\pi}_0^{\A^1}\!(X)$ is an equivalence. 
\end{de}
\begin{ex}\label{A^1-invariant sheaves are discrete}
Any $\A^1$-invariant Nisnevich sheaf $M$ of discrete spaces on $\Sm_k$ is $\A^1$-discrete. Indeed, under these assumptions the Yoneda lemma tells us that $M$ agrees with $[-,M] \cong \widetilde{\phantom{\pi}}\hspace{-6pt}\pi_0^{\A^1}\!(M).$ Since $M$ is a sheaf, this presheaf must already be one as well and thus coincides with its sheafification $\underline{\pi}_0^{\A^1}\!(M)$. In particular, $\G_{\rm m}$ is $\A^1$-discrete.
\end{ex}
\begin{ex}\label{Gm x A1 is discrete}
Also $\G_{\rm m} \times \A^1$ is $\A^1$-discrete. For this note that he presheaf 
$$\widetilde{\phantom{\pi}}\hspace{-6pt}\pi_0^{\A^1}\!(-)(U) = [S^0 \wedge U_+,-]_* \cong [U_+,-]_* \cong [U,-]$$
is $\A^1$-invariant, so that we have $\widetilde{\phantom{\pi}}\hspace{-6pt}\pi_0^{\A^1}\!(\G_{\rm m} \times \A^1) \cong \widetilde{\phantom{\pi}}\hspace{-6pt}\pi_0^{\A^1}\!(\G_{\rm m})$ and hence $\underline{\pi}_0^{\A^1}\!(\G_{\rm m} \times \A^1) \cong \underline{\pi}_0^{\A^1}\!(\G_{\rm m})$ via the projection map. Therefore, using the $\A^1$-discreteness of $\G_{\rm m}$ and that the projection $\G_{\rm m} \times \A^1 \rightarrow \G_{\rm m}$ is an equivalence, also $\G_{\rm m} \times \A^1$ is $\A^1$-discrete.
\end{ex}
The reason that we argued as we just did is that the sheaves $\underline{\pi}_0^{\A^1}\!(X)$ are not $\A^1$-invariant in general. The $\A^1$-invariance of the sheaves $\underline{\pi}_0^{\A^1}\!(X)$ was conjectured by Morel (for instance as Conjecture 1.12 in \cite{MR2934577}), but a counterexample was recently discovered by Ayoub \cite{Pizero}. 
\begin{lem}\label{Homotopy of Discrete Motivic Spaces}
Let $X$ be a pointed $\A^1$-discrete motivic space. Then $\underline{\pi}^{\A^1}_n\!(X) = 0$ for all positive $n$.
\end{lem}
\begin{proof}
Since $X$ is $\A^1$-discrete, we have $X \simeq \underline{\pi}^{\A^1}_0\!(X)$, so that it suffices to prove the statement for $\underline{\pi}^{\A^1}_0\!(X)$. This is a sheaf of discrete spaces so that $\underline{\pi}_n(\underline{\pi}^{\A^1}_0\!(X)) = 0$ for all $n \geq 1$. Via the equivalence from above, $\underline{\pi}^{\A^1}_0\!(X)$ is $\A^1$-invariant, so that $\underline{\pi}_n(\underline{\pi}^{\A^1}_0\!(X)) = \underline{\pi}_n^{\A^1}\!(\underline{\pi}^{\A^1}_0\!(X))$. Therefore we are done.
\end{proof}
Also many other notions and tools, such as connectedness and homology, work in this setting. We do not want to introduce these here, but let us note that they are similar to their counterparts in topology; see Section 6 of \cite{MR2934577}. Therefore we hope that also readers not yet familiar with these should have a general understanding of how to deal with these notions.
\begin{ex}\label{Homotopy Sheaf Example}
Let $p \geq 2$. The sphere $S^p \in \Spc_*$ is $(p-1)$-connected, so that by Theorem 1.18 of \cite{MR2934577}, $S^{p+q,q} = S^p \wedge \G_{\rm m}^{\wedge q}$ is $\A^1$-$(p-1)$-connected. Using the Hurewicz theorem (\cite[Theorem 6.37]{MR2934577}) and suspension isomorphisms (\cite[Remark 6.30]{MR2934577}) of $\A^1$-homology, we thus have
$$\underline{\pi}_p^{\A^1}\!(S^{p+q,q}) = \underline{\pi}_p^{\A^1}\!(S^p \wedge \G_{\rm m}^{\wedge q}) \cong \widetilde{\phantom{H}}\hspace{-9.5pt}\underline{H}_p^{\A^1}\!(S^p \wedge \G_{\rm m}^{\wedge q}) \cong \widetilde{\phantom{H}}\hspace{-9.5pt}\underline{H}_0^{\A^1}\!(\G_{\rm m}^{\wedge q}).$$
By Example \ref{Punctured Affine Space as Sphere}, we have $\A^n \setminus \lbrace 0 \rbrace \simeq S^{2n-1,n}$, so that this in particular shows
$$\underline{\pi}_{n-1}^{\A^1}(\A^n \setminus \lbrace 0 \rbrace) \cong \widetilde{\phantom{H}}\hspace{-9.5pt}\underline{H}_0^{\A^1}\!(\G_{\rm m}^{\wedge n})$$ 
for all $n \geq 3$. This is also true for $n =2$, but the argument is a tiny bit more involved. The map $\SL_2 \rightarrow \A^2 \setminus \lbrace 0 \rbrace$ which maps a matrix to its last column is an equivalence, so that $\A^2 \setminus \lbrace 0 \rbrace$ is a motivic version of an H-space. As is true classically, the fundamental sheaf of such motivic spaces is abelian, and in this case a weak form of the Hurewicz theorem (see Remark \ref{Motivic Hurewicz in deg 1}) applies and gives the isomorphism to homology; see also Remark 6.41 in \cite{MR2934577}. 

For $q \neq 0$, these are computations of certain motivic stable homotopy sheaves of spheres in terms of some kind of free abelian object on $\G_{\rm m}^{\wedge q}$, namely the so-called free strictly $\A^1$-invariant sheaf on the pointed motivic space $\G_{\rm m}^{\wedge q}$. There is, of course, a precise meaning behind those words, but for this article it suffices to treat this as a name, which we therefore do. If the reader wants to learn more about those notion, we refer them to \cite{TomWS}.
\end{ex}
Let us try to understand the right hand side, that is, the reduced $\A^1$-homology sheaf $\widetilde{\phantom{H}}\hspace{-9.5pt}\underline{H}_0^{\A^1}\!(\G_{\rm m}^{\wedge q})$, a bit better. For $q = 0$ it can easily be computed since $\G_{\rm m}^{\wedge 0} \simeq S^0$, and the result is the constant sheaf $\underline{\Z}$. For $q > 0$, we need the following.
\begin{de}\label{Def of KMW}
The (naive) Milnor--Witt K-theory sheaf $\KK^{\MW}_*$ is the sheaf of $\Z$-graded unital associative rings generated by $\G_{\rm m}$ in degree $1$ and by a symbol $\eta$ in degree $-1$, subject to:
\begin{enumerate}
\item[(i)] $[ab] = [a] + [b] + \eta[a][b]$, $a,b \in \mathcal{O}(-)^\times$ \hfill (Twisted tensor relation)
\item[(ii)] $[a][1-a] = 0$, $a, 1-a \in \mathcal{O}(-)^\times$ \hfill (Steinberg relation)
\item[(iii)] $\eta[a] = [a]\eta$, $a \in \mathcal{O}(-)^\times$ \hfill (Centrality relation)
\item[(iv)] $\eta(2 + \eta[-1]) = 0$ \hfill (Witt relation)
\end{enumerate}
\end{de}
There is actually a little subtlety here. For finite fields the above definition is not the correct one. The object we actually want is called unramified Milnor--Witt K-theory, which is quite a bit more cumbersome to define as a sheaf; see Section 2.3 and 3.2 of \cite{MR2934577}. To keep this recollection simpler, will not get into the details and just present this simpler definition which works over most fields. Morel shows in Theorem 3.37 of \cite{MR2934577}:

\begin{thm}\label{Milnor--Witt K-theory as Free Object}
Let $q > 0$. The universal symbol map $\G_{\rm m}^{\wedge q} \rightarrow \KK^{\MW}_q$, $(a_1,\dotsc,a_q) \mapsto [a_1]\cdot \dotsc \cdot [a_q]$ realizes $\KK^{\MW}_q$ as the free strictly $\A^1$-invariant sheaf on the pointed motivic space $\G_{\rm m}^{\wedge q}$.
\end{thm}
In other words, we have $\widetilde{\phantom{H}}\hspace{-9.5pt}\underline{H}_0^{\A^1}\!(\G_{\rm m}^{\wedge q}) \cong \KK^{\MW}_q$ for $q > 0$, and the above map shows us how to include $\G_{\rm m}^{\wedge q}$ as generators of this free object.

As strict $\A^1$-invariance is a stronger notion than $\A^1$-invariance, Example \ref{A^1-invariant sheaves are discrete} yields:
\begin{cor}\label{KMW discrete}
For all $q > 1$, the sheaves $\KK^{\MW}_q$ are discrete motivic spaces.
\end{cor}
The restriction to positive $q$ is actually not needed, but for non-positive $q$ we would need to argue differently; see Section 2.3 and 3.2 of \cite{MR2934577}. Moreover, as $q = 2$ is the most relevant case for this article, the above is sufficient.

\section{Knot Theory from the Motivic point of view}\label{Section 4}
In this section we present the main content of this article, which is to use motivic homotopy theory to lift some knot theoretic ideas to algebraic geometry. Therefore, let us remind the reader of some standard facts of knot theory, which can, for instance, be found in \cite{MR3156509}.

\hfill

There are quite some knot invariants which are defined in terms of the associated knot complements. For instance, given a knot $i \colon S^1 \hookrightarrow S^3$ with complement $X = S^3 \setminus i(S^1)$, we can consider the fundamental group $\pi_1(X)$, which is known as the knot group. Its abelianization is, independently of the knot, given by 
$$\pi_1(X)^{\ab} \cong H_1(X) \cong \Z$$
and is hence not an interesting invariant. But there is a way to turn it into a rather powerful invariant by
using the Galois correspondence from topology, that is, the correspondence between subgroups of $\pi_1(X)$ and coverings of $X$. Associated to the commutator subgroup $[\pi_1(X), \pi_1(X)] \subset \pi_1(X)$ is the universal $\Z$-covering $p \colon \widetilde{X} \rightarrow X$ of the knot complement, and its first homology $H_1(\widetilde{X})$ is the so-called Alexander module of the knot. It is a module over the group ring $\Z[C_\infty] = \Z[t,t^{-1}]$, where $C_\infty = \Z$ acts via deck transformations. Using local coefficients, the Alexander module can also be expressed as
$$H_1(\widetilde{X}) \cong H_1(\widetilde{X}, \underline{\Z}) \cong H_1(X,p_*\underline{\Z}),$$
which is a trade-off. The space is simpler, but the coefficients are more complicated.

\hfill

Based on this we can now present an overview of the motivic story.
\vspace*{-15pt}
\begin{center}
\def\arraystretch{1.5}
\begin{tabular}{ p{7cm} p{7cm} }\label{Table}

$\phantom{1}$ & $\phantom{1}$ \\

Classical theory & Motivic theory \\

\hline
knot $i \colon S^1 \hookrightarrow S^3$ & (algebraic) knot $i \colon \A^1 \hookrightarrow \A^3$ \\
knot complement $X = S^3 \setminus i(S^1)$ & knot complement $X = \A^3 \setminus \hspace{2pt} i(\A^1)$ \\
knot group $\pi_1(X)$ & knot sheaf $\underline{\pi}_1^{\A^1}\!(X)$ \\
abelianization $H_1(X) \cong \Z$ & abelianization? $\underline{H}_1^{\A^1}\!(X) \cong \KK^{\MW}_2$ \\
universal abelian $\Z$-covering $p \colon \widetilde{X} \rightarrow X$ & universal abelian $\KK^{\MW}_2$-covering $p \colon \widetilde{X} \rightarrow X$ \\
Alexander module $H_1(\widetilde{X})$ & (motivic) Alexander module $\underline{H}_1^{\A^1}\!(\widetilde{X})$
\end{tabular}
\end{center}
\begin{rem}\label{Motivic Hurewicz in deg 1}
Given a pointed $\A^1$-connected motivic space $X$, the Hurewicz map $h \colon \underline{\pi}_1^{\A^1}\!(X) \rightarrow \underline{H}_1^{\A^1}\!(X)$ is not fully understood. By Theorem 6.35 of \cite{MR2934577}, it is the initial morphism to a strictly $\A^1$-invariant sheaf, which is referred to as the weak Hurewicz theorem. Choudhury and Hogadi \cite{MR4442407} claimed that it is an epimorphism, but their argument is unfortunately known to be incomplete; see, for instance, \cite{zbMATH08029975}. The kernel of $h$ is also not yet known to be the commutator subsheaf, but all of this is expected to be true (potentially after $\A^1$-localizing again or for a slightly different notion of abelianization).
\end{rem}
Before we discuss $\KK^{\MW}_2$-coverings and the associated Alexander modules, let us explain why exactly $\KK^{\MW}_2$ is the replacement for $\Z$ in the motivic story, i.e. why there is an isomorphism
$$\underline{H}_1^{\A^1}\!(X) \cong \KK^{\MW}_2.$$
As mentioned before, we do not expect the reader to know much about about $\A^1$-homology. Any reader willing to believe that $\A^1$-homology works similarly enough as ordinary homology of spaces should, despite their potential unfamiliarity with this notion, still be able to understand most of our arguments.
\begin{prop}\label{Homology of Knot Complement}
The first $\A^1$-homology sheaf of the complement of any algebraic knot $i \colon \A^1 \hookrightarrow \A^3$ is isomorphic to $\KK^{\MW}_2$.
\end{prop}
\begin{proof}
First we consider the unknot $u \colon \A^1 \hookrightarrow \A^3$, i.e., an embedding of $\A^1$ as one of the three coordinate axes. As the projection $\A^3 \setminus u(\A^1) \rightarrow \A^2 \setminus \lbrace 0 \rbrace$ is a trivial line bundle, we see that the knot complement $X = \A^3 \setminus u(\A^1)$ is equivalent to $\A^2 \setminus \lbrace 0 \rbrace$, which by Example \ref{Punctured Affine Space as Sphere} is equivalent to the sphere $S^{3,2}$. Therefore, we are done by Example \ref{Homotopy Sheaf Example} together with Theorem \ref{Milnor--Witt K-theory as Free Object}.

Now let $i \colon \A^1 \hookrightarrow \A^3$ be any knot and let $j \colon \A^3 \setminus i(\A^1) \hookrightarrow \A^3$ be the embedding of the knot complement. We then consider the cofiber sequence
\begin{center}
\begin{tikzcd}
\A^3 \setminus i(\A^1) \arrow[r, hookrightarrow, "j"] & \A^3 \arrow[r] & \A^3/(\A^3 \setminus i(\A^1)) \arrow[r] & \Sigma \A^3 \setminus i(\A^1)
\end{tikzcd}
\end{center}
in $\Spc(k)$. Homotopy purity (Theorem 2.23 of \cite{MR1813224}) tells us that the cofiber $\A^3 / (\A^3 \setminus i(\A^1))$ is equivalent to the Thom space $\Th(\mathcal{N}_{\A^1/\A^3})$ of the normal bundle associated with the embedding $i \colon \A^1 \hookrightarrow \A^3$. By the Quillen--Suslin theorem (\cite{MR427303} and \cite{MR469905}) or the fact that $\A^1$ is the spectrum of a principal ideal domain, we know that the normal bundle $\mathcal{N}_{\A^1\!/\A^3}$ is a trivial bundle, and Thom spaces of trivial bundles are not difficult to compute. Indeed,
$$\Th(\mathcal{O}^n_X) = (\A^n \times \,X)/(\A^n \setminus \lbrace 0 \rbrace \times X) \simeq \A^n / (\A^n \setminus \lbrace 0 \rbrace) \wedge X_+ \simeq S^{2n,n} \wedge X_+$$
for any smooth scheme $X$ and trivial rank $n$ bundle $\mathcal{O}_X^n$. All in all, we thus have 
$$\cofib(j) = \A^3/(\A^3 \setminus i(\A^1)) \simeq \Th(\mathcal{N}_{\A^1\!/\A^3}) \simeq \Th(\mathcal{O}^2_{\A^1}) \simeq S^{4,2} \wedge \A^1_+ \simeq S^{4,2} \wedge S^0 \simeq S^{4,2}.$$
Using this, if we now consider the long exact sequence in reduced homology 
\begin{center}
\begin{tikzcd}
\dotsc \arrow[r] & \widetilde{\phantom{H}}\hspace{-9.5pt}\underline{H}_2^{\A^1}\!(\A^3) \arrow[r] & \widetilde{\phantom{H}}\hspace{-9.5pt}\underline{H}_2^{\A^1}\!(S^{4,2}) \arrow[r] & \widetilde{\phantom{H}}\hspace{-9.5pt}\underline{H}_1^{\A^1}\!(\A^3 \setminus i(\A^1)) \arrow[r] & \widetilde{\phantom{H}}\hspace{-9.5pt}\underline{H}_1^{\A^1}\!(\A^3) \arrow[r] & \dotsc
\end{tikzcd}
\end{center}
associated to the above cofiber sequence, we obtain $\widetilde{\phantom{H}}\hspace{-9.5pt}\underline{H}_2^{\A^1}\!(S^{4,2}) \cong \widetilde{\phantom{H}}\hspace{-9.5pt}\underline{H}_1^{\A^1}\!(\A^3 \setminus i(\A^1))$ since affine spaces are contractible and hence have no reduced homology. We just observed that the first homology of the knot complement is independent of the knot. As we know that it is given by $\KK^{\MW}_2$ for the unknot, it is thus $\KK^{\MW}_2$ for all knots.
\end{proof}
Now we construct the desired $\A^1$-coverings of algebraic knot complements.
\begin{con}\label{Construction of Milnor--Witt K_2 coverings}
We first construct a $\KK^{\MW}_2$-covering for the unknot $u \colon \A^1 \hookrightarrow \A^3$. As observed in the proof of Proposition \ref{Homology of Knot Complement}, its complement is equivalent to $\A^2 \setminus \lbrace 0 \rbrace$. Now we consider the standard open cover 
$$\A^2 \setminus \lbrace 0 \rbrace = (\A^1 \times \G_{\rm m}) \cup (\G_{\rm m} \cup \A^1)$$
with intersection $\G_{\rm m} \times \G_{\rm m}$, which gives us a pushout square
\begin{center}
\begin{tikzcd}
\G_{\rm m} \times \G_{\rm m} \arrow[r, hookrightarrow] \arrow[d, hookrightarrow] & \G_{\rm m} \times \A^1 \arrow[d, hookrightarrow] \\
\A^1 \times \G_{\rm m} \arrow[r, hookrightarrow] & \A^2 \wo \lbrace 0 \rbrace. \arrow[ul, phantom, "\ulcorner", very near start]
\end{tikzcd}
\end{center}
By the previous lemma, the motivic spaces $\A^1 \times \G_{\rm m}$ and $\G_{\rm m} \times \A^1$ are $\A^1$-discrete, so any covering of $\A^2 \setminus \lbrace 0 \rbrace$, including the one we want to construct, must trivialize on $\A^1 \times \G_{\rm m}$ and $\G_{\rm m} \times \A^1$. Therefore, we consider two trivial $\KK^{\MW}_2$-coverings
$$\A^1 \times \G_{\rm m} \times \KK^{\MW}_2 \rightarrow \A^1 \times \G_{\rm m} \hspace{5pt}\text{ and }\hspace{5pt} \G_{\rm m} \times \A^1 \times \KK^{\MW}_2\rightarrow \G_{\rm m} \times \A^1$$
and glue these over $\G_{\rm m} \times \G_{\rm m}$. For the gluing map we choose the universal symbol map
$$\G_{\rm m} \times \G_{\rm m} \times \KK^{\MW}_2 \rightarrow \G_{\rm m} \times \G_{\rm m} \times \KK^{\MW}_2, (a,b,x) \mapsto (a,b,x - [a][b]).$$
This gives us the desired $\KK^{\MW}_2$-covering $\reallywidetilde{\A^2 \setminus \lbrace 0 \rbrace} \rightarrow \A^2 \setminus \lbrace 0 \rbrace \simeq \A^3 \setminus u(\A^1)$.

Now let $i \colon \A^1 \hookrightarrow \A^3$ be an arbitrary algebraic knot. By Theorem \ref{Knots are Complete Intersections}, there exist polynomials $f,g \in k[x,y,z]$ with $i(\A^1) = V(f,g)$. This gives rise to the map
$$(f,g) \colon \A^3 \setminus i(\A^1) = \A^3 \setminus V(f,g) \rightarrow \A^2 \setminus \lbrace 0 \rbrace$$
and we can consider the motivic space $\reallywidetilde{\A^3 \setminus i(\A^1)}$ defined via the pullback square
\begin{center}
\begin{tikzcd}
\reallywidetilde{\A^3 \setminus i(\A^1)} \arrow[r] \arrow[d] \arrow[dr, phantom, "\lrcorner", very near start] & \reallywidetilde{\A^2 \setminus \lbrace 0 \rbrace} \arrow[d] \\
\A^3 \setminus i(\A^1) \arrow[r, "(f \, \komma g)"] & \A^2 \setminus \lbrace 0 \rbrace,
\end{tikzcd}
\end{center}
where the right vertical map is the $\KK^{\MW}_2$-covering that we already constructed. As pullbacks of coverings are coverings, this gives us the desired $\KK^{\MW}_2$-covering of $\A^3 \setminus i(\A^1)$. In Theorem \ref{Coverings are Universal} we will see that these do not depend on the choice of the polynomials $f$ and $g$, which is why we allow ourselves to drop this potential dependence from the notation.
\end{con}
From the construction it is certainly not clear that these coverings are the universal abelian ones. Actually, since the Hurewicz map in degree $1$ is not yet known to be the abelianization (see also Remark \ref{Motivic Hurewicz in deg 1}), we define this notion as follows.
\begin{de}
Let $X$ be an $\A^1$-connected space over $k$. An abelian covering $p \colon \widetilde{X} \to X$ is called universal, if $\pi_1(\widetilde{X})$ is the kernel of the Hurewicz map $h \colon \underline{\pi}_1^{\A^1}\!(X) \to \underline{H}_1^{\A^1}\!(X)$.
\end{de}
Here we use that the short exact sequence 
\begin{center}
\begin{tikzcd}
0 \arrow[r] & \underline{\pi}_1^{\A^1}\!(\widetilde{X}) \arrow[r, "p_*"] & \underline{\pi}_1^{\A^1}\!(X) \arrow[r] & \underline{\Deck}_X(\widetilde{X}) \arrow[r] & 0
\end{tikzcd}
\end{center}
associated to any abelian covering $p \colon \widetilde{X} \to X$ identifies $\underline{\pi}_1^{\A^1}\!(\widetilde{X})$ with the subsheaf $p_*\underline{\pi}_1^{\A^1}\!(\widetilde{X}) \hookrightarrow \underline{\pi}_1^{\A^1}\!(X)$. If the Hurewicz map is the abelianization, as is expected, an abelian covering is universal if it corresponds to the commutator subsheaf.

Now that we know what exactly we want to verify, let us start by dealing with the connectivity of the total spaces and the base spaces. To do so, we prove that the Hurewicz map is an epimorphism for our objects of interest, at least under additional assumptions:
\begin{thm}\label{Hurewicz}
If the base field $k$ is algebraically closed, the Hurewicz map $h \colon \underline{\pi}_1^{\A^1}\!(X)\to \underline{H}_1^{\A^1}\!(X)$ is an epimorphism for all algebraic knot complements $X$.
\end{thm}
Although it is non-trivial, we will postpone the proof for now as we believe that it might distract the reader from the core arguments of this article. Instead, we present the proof in Appendix \ref{Appendix B}. For all the explicit algebraic knots from Example \ref{Knot picture} (in fact, for all algebraic knots defined by Chebyshev polynomials), this theorem is also true for when the base field is $\R$; see Remark \ref{rem:assumptions for hurewicz}.

\begin{prop}\label{knot complements connected}
Let $p \colon \widetilde{X} \rightarrow X$ be the $\KK^{\MW}_2$-covering of an algebraic knot complement $X$ from Construction \ref{Construction of Milnor--Witt K_2 coverings}. Then $X$ is $\A^1$-connected. If the Hurewicz map $h \colon \underline{\pi}_1^{\A^1}\!(X)\to \underline{H}_1^{\A^1}\!(X)$ is an epimorphism, then $\widetilde{X}$ is also $\A^1$-connected.
\end{prop}
\begin{proof}
Since any algebraic knot $i \colon \A^1 \hookrightarrow \A^3$ is of codimension $2$, Asok and Doran's homotopy excision \cite[Theorem~4.1]{asok:doran:excision} yields that the map
$$\underline{\pi}_0^{\A^1}\!(X) \to \underline{\pi}_0^{\A^1}\!(\A^3) \cong 0$$
induced by the inclusion is an isomorphism. In other words, $X$ is $\A^1$-connected.

\hfill

Now we deal with the $\A^1$-connectivity of $\widetilde{X}$. Given an algebraic knot $i \colon \A^1 \hookrightarrow \A^3$, Theorem \ref{Knots are Complete Intersections} lets us find polynomials $f,g \in k[x,y,z]$ with $i(\A^1) = V(f,g)$, from which we obtain the map
$$\varphi = (f,g) \colon X = \A^3 \setminus i(\A^1) = \A^3 \setminus V(f,g) \rightarrow \A^2 \setminus \lbrace 0 \rbrace$$
as in Construction \ref{Construction of Milnor--Witt K_2 coverings}. This allows us to consider the pullback square
\begin{center}
\begin{tikzcd}
X \arrow[r, hookrightarrow, "j"] \arrow[d, swap, "\varphi"] & \A^3 \arrow[d,"(f \, \komma g)"] \\
\A^2\setminus\lbrace 0\rbrace \arrow[r, hookrightarrow, "j_0"] & \A^2.
\arrow[ul, phantom, "\lrcorner", very near end]
\end{tikzcd}
\end{center}
Using homotopy purity (Theorem 2.23 of \cite{MR1813224}), its naturality (for instance, Theorem 3.23 of \cite{MR3570135}) and $\A^1$-invariance, we get a commutative diagram
\begin{center}
\begin{tikzcd}
\A^3 / X \arrow[r, "\simeq"] \arrow[d, swap, "(f \, \komma g)"] & \Th(\mathcal{N}_{\A^1\!/\!\A^3}) \arrow[d] \arrow[r, "\simeq"] \arrow[d] & \Th(\mathcal{O}^2_{\A^1}) \arrow[r, "\simeq"] \arrow[d,] & S^{4,2} \wedge \A^1_+ \arrow[d,"\simeq"]\\
\A^2 / (\A^2 \setminus \lbrace 0 \rbrace) \arrow[r, "\simeq"] & \Th(\mathcal{N}_{\lbrace 0 \rbrace /\!\A^2}) \arrow[r, "\simeq"] & \Th(\mathcal{O}^2_{\lbrace 0 \rbrace}) \arrow[r, "\simeq"] & S^{4,2} \wedge \lbrace 0 \rbrace_+
\end{tikzcd}
\end{center}
which yields an equivalence between the cofibers of $j$ and $j_0$. Here we also use Proposition \ref{Knots are Complete Intersections}, i.e., that algebraic knots are complete intersections, to get commutativity of the middle square, and we once again make use of the fact that it is simple to compute Thom spaces of trivial bundles (as in the proof of Proposition \ref{Homology of Knot Complement}). Therefore, we obtain a morphism of cofiber/Puppe sequences
\begin{center}
\begin{tikzcd}
X \arrow[r, hookrightarrow, "j"] \arrow[d, swap, "\varphi"] & \A^3 \arrow[d, swap, "(f \, \komma g)"] \arrow[r] \arrow[d] & \A^3 / X \arrow[r] \arrow[d,"\simeq"] & \Sigma X  \arrow[r] \arrow[d,"\Sigma\varphi"] & \Sigma \A^3 \arrow[d,"\Sigma(f \, \komma g)"]\\
\A^2 \setminus \lbrace 0 \rbrace \arrow[r, hookrightarrow, "j_0"] & \A^2 \arrow[r] & \A^2 / (\A^2 \setminus \lbrace 0 \rbrace) \arrow[r] & \Sigma \! \A^2 \setminus \lbrace 0 \rbrace  \arrow[r] & \Sigma \! \A^2.
\end{tikzcd}
\end{center}
Since affine spaces and hence also their suspensions are contractible, the maps $\A^3 / X \to \Sigma X$ and $\A^2 / (\A^2 \setminus \lbrace 0 \rbrace) \to \Sigma \! \A^2 \setminus \lbrace 0 \rbrace$ are equivalences, so that consequently also the map $\Sigma \varphi \colon \Sigma X \to \Sigma \! \A^2 \setminus \lbrace 0 \rbrace$ is an equivalence. The suspension isomorphism then yields an isomorphism $\varphi_*\colon H_1^{\A^1}\!(X)\to H_1^{\A^1}\!(\A^2\setminus\lbrace 0\rbrace)$. From the functoriality of the Hurewicz morphism, we get a commutative diagram
\begin{center}
\begin{tikzcd}
\underline{\pi}_1^{\A^1}\!(X) \arrow[r,"\varphi_*"] \arrow[d, swap, "h"] & \underline{\pi}_1^{\A^1}\!(\A^2\setminus\lbrace 0\rbrace) \arrow[d,"\cong"] \\
\underline{H}_1^{\A^1}\!(X) \arrow[r,"\varphi_*"]  & \underline{H}_1^{\A^1}\!(\A^2\setminus \lbrace 0\rbrace)
\end{tikzcd}
\end{center}
where the vertical isomorphism on the right is the one from Example \ref{Homotopy Sheaf Example}. If the Hurewicz map $h \colon \underline{\pi}_1^{\A^1}\!(X)\to \underline{H}_1^{\A^1}\!(X)$ is an epimorphism then so is  $\varphi_*\colon\underline{\pi}_1^{\A^1}\!(X) \to \underline{\pi}_1^{\A^1}\!(\A^2\setminus\lbrace 0\rbrace)$ as a composition of epimorphisms. If we now consider the Mayer-Vietoris sequence 
\begin{center}
\begin{tikzcd}
\dotsc \arrow[r] & \underline{\pi}_1^{\A^1}\!(\widetilde{X}) \arrow[r] & \underline{\pi}_1^{\A^1}\!(\reallywidetilde{\A^2 \setminus \lbrace 0 \rbrace}) \times \underline{\pi}_1^{\A^1}\!(X) \arrow[r] & \underline{\pi}_1^{\A^1}\!(\A^2 \setminus \lbrace 0 \rbrace) \arrow[r] & \underline{\pi}_0^{\A^1}\!(\widetilde{X}) \arrow[r] & \dotsc.
\end{tikzcd}
\end{center}
associated with the pullback square
\begin{center}
\begin{tikzcd}
\widetilde{X} \arrow[r] \arrow[d] \arrow[dr, phantom, "\lrcorner", very near start] & \reallywidetilde{\A^2 \setminus \lbrace 0 \rbrace} \arrow[d] \\
X \arrow[r] & \A^2 \setminus \lbrace 0 \rbrace
\end{tikzcd}
\end{center}
defining $p$, then the exactness together with $\varphi_*$ being an epi yields the claim.
\end{proof}
As could be expected, the covering associated with the unknot is not only the universal abelian covering, but even the universal one:
\begin{lem}\label{Unknot is Univ Covering}
The $\KK^{\MW}_2$-covering of the complement of the unknot from Construction \ref{Construction of Milnor--Witt K_2 coverings} is universal.
\end{lem}
\begin{proof}
This is a version of Remark 7.18 (1) of \cite{MR2934577}, for which we provide some details. We can view the universal $\KK^{\MW}_2$-covering of $\A^2\setminus \lbrace 0 \rbrace$ as a $\KK^{\MW}_2$-torsor, which is therefore classified by an element in $H^1(\A^2 \setminus \lbrace 0 \rbrace, \KK^{\MW}_2)$. Any such element restricts trivially to $\Gm \times \A^1$ and $\A^1 \times \Gm$ as these motivic spaces are $\A^1$-discrete by Example \ref{Gm x A1 is discrete}. This means that any torsor will trivialize in the natural cover of $\A^2 \setminus \lbrace 0\rbrace$, and it remains to identify the corresponding cocycle $\Gm \times \Gm \to \KK^{\MW}_2$. Since $\A^2\setminus\lbrace 0\rbrace\simeq \Sigma\Gm\wedge\Gm$, we can use the suspension isomorphism to identify
\[
H^1(\A^2\setminus\lbrace 0\rbrace,\KK^{\MW}_2)\cong H^0(\Gm\wedge\Gm,\KK^{\MW}_2),
\]
which maps the class of a $\KK^{\MW}_2$-torsor on the left-hand side to a map $\Gm\wedge \Gm\to \KK^{\MW}_2$ on the right-hand side, from which we obtain the corresponding cocycle $\Gm \times \Gm \to \Gm \wedge \Gm \to \KK^{\MW}_2$ via composition with the quotient map $\Gm \times \Gm \to \Gm \wedge \Gm$. 

Finally, note that the universal symbol $\Gm\wedge\Gm\to\KK^{\MW}_2$ corresponds to the universal $\KK^{\MW}_2$-torsor over $\A^2\setminus\lbrace 0\rbrace$. Indeed, under the suspension-loops adjunction $[\A^2\setminus\lbrace 0 \rbrace, K(\KK^{\MW}_2,1)] \cong [\Gm \wedge \Gm, \KK^{\MW}_2]$, the universal symbol corresponds to a map $\alpha\colon\A^2\setminus\lbrace 0\rbrace\to K(\KK^{\MW}_2,1)$ which induces the isomorphism of Example~\ref{Homotopy Sheaf Example}. Consequently, the universal covering of $\A^2\setminus\lbrace 0\rbrace$ is given by Construction~\ref{Construction of Milnor--Witt K_2 coverings}.
\end{proof}
\begin{thm}\label{Coverings are Universal}
The $\KK^{\MW}_2$-coverings $p \colon \widetilde{X} \to X$ from Construction \ref{Construction of Milnor--Witt K_2 coverings} are universal abelian, provided that the Hurewicz map $h \colon \underline{\pi}_1^{\A^1}\!(X)\to \underline{H}_1^{\A^1}\!(X)$ is an epimorphism.
\end{thm}
Before turning to the proof, let us note that this already shows that the $ \A^1$-coverings are independent of the choice of polynomials $f,g \in k[x,y,z]$ used in Construction \ref{Construction of Milnor--Witt K_2 coverings} to realize $i(\A^1) = V(f,g)$ as a complete intersection.
\begin{proof}
Let $i \colon \A^1 \hookrightarrow \A^3$ be an arbitrary algebraic knot with knot complement $X$ and with $\KK^{\MW}_2$-covering $p \colon \widetilde{X} \rightarrow X$ as in Construction \ref{Construction of Milnor--Witt K_2 coverings}. From Proposition \ref{knot complements connected} we already know that both $X$ and $\widetilde{X}$ are $\A^1$-connected. Now the previous lemma tells us that the $\KK^{\MW}_2$-covering $\reallywidetilde{\A^2 \setminus \lbrace 0 \rbrace} \rightarrow \A^2 \setminus \lbrace 0 \rbrace$ is universal and hence also abelian. As pullbacks of abelian coverings are abelian, this shows that $p \colon \widetilde{X} \rightarrow X$ is abelian. In particular, there is a short exact sequence
\begin{center}
\begin{tikzcd}
0 \arrow[r] & \underline{\pi}_1^{\A^1}\!(\widetilde{X}) \arrow[r, "p_*"] & \underline{\pi}_1^{\A^1}\!(X) \arrow[r] & \underline{\Deck}_X(\widetilde{X}) \arrow[r] & 0
\end{tikzcd}
\end{center}
with $[\underline{\pi}_1^{\A^1}\!(X),\underline{\pi}_1^{\A^1}\!(X)] \subset p_*\underline{\pi}_1^{\A^1}\!(\widetilde{X})$. Since $\underline{\pi}_0^{\A^1}\!(\widetilde{X}) = 0$ and $\KK^{\MW}_2$ is $\A^1$-discrete by Corollary \ref{KMW discrete}, the long exact sequence of homotopy sheaves associated with $p$ together with Lemma \ref{Homotopy of Discrete Motivic Spaces} yields the short exact sequence
\begin{center}
\begin{tikzcd}
0 \arrow[r] & \underline{\pi}_1^{\A^1}\!(\widetilde{X}) \arrow[r, "p_*"] & \underline{\pi}_1^{\A^1}\!(X) \arrow[r] & \KK^{\MW}_2 \arrow[r] & 0.
\end{tikzcd}
\end{center}
By the universal property of cokernels we thus have $\underline{\Deck}_X(\widetilde{X}) \cong \KK^{\MW}_2$. Proposition \ref{Homology of Knot Complement} now gives us the following diagram:
\begin{center}
\begin{tikzcd}
\underline{\pi}_1^{\A^1}\!(X) \arrow[r, twoheadrightarrow] \arrow[d, "h"]& \underline{\pi}_1^{\A^1}\!(X)/[\underline{\pi}_1^{\A^1}\!(X),\underline{\pi}_1^{\A^1}\!(X)] \arrow[r, twoheadrightarrow] & \underline{\pi}_1^{\A^1}\!(X)/p_*\underline{\pi}_1^{\A^1}\!(\widetilde{X}) \arrow[r, "\cong"] & \underline{\Deck}_X(\widetilde{X}) \arrow[r, "\cong"] & \KK^{\MW}_2 \\
\underline{H}_1^{\A^1}\!(X)  \\ \KK^{\MW}_2 \arrow[u, "\cong"]
\end{tikzcd}
\end{center}
Note that except for the sheaf on the top left, all the sheaves in this diagram are abelian. We can make all of those strongly $\A^1$-invariant by enforcing $\A^1$-invariance on the level of their classifying spaces; see Remark 1.8 (2). For the abelian ones, this automatically yields strictly $\A^1$-invariant sheaves by Theorem 5.46 of \cite{MR2934577}. We denote this strong $\A^1$-localization by $L_{\A^1}$. Since $\A^1$-fundamental groups are strongly $\A^1$-invariant as shown in Theorem 6.1 of \cite{MR2934577}, the weak Hurewicz theorem \cite[Theorem 6.35]{MR2934577} can be applied to $$\underline{\pi}_1^{\A^1}\!(X) \cong L_{\A^1}\underline{\pi}_1^{\A^1}\!(X) \twoheadrightarrow L_{\A^1}\bigl(\underline{\pi}_1^{\A^1}\!(X)/[\underline{\pi}_1^{\A^1}\!(X),\underline{\pi}_1^{\A^1}\!(X)]\bigr),$$
which is an epimorphism since $L_{\A^1}$ is a left adjoint. 
Using that various other sheaves in our diagram from above are already strongly $\A^1$-invariant, this gives the commutative diagram
\begin{center}
\begin{tikzcd}
\underline{\pi}_1^{\A^1}\!(X) \arrow[r, twoheadrightarrow] \arrow[d, "h"]& L_{\A^1}\bigl(\underline{\pi}_1^{\A^1}\!(X)/[\underline{\pi}_1^{\A^1}\!(X),\underline{\pi}_1^{\A^1}\!(X)]\bigr) \arrow[r, twoheadrightarrow] & L_{\A^1}\bigl(\underline{\pi}_1^{\A^1}\!(X)/p_*\underline{\pi}_1^{\A^1}\!(\widetilde{X})\bigr) \arrow[r, "\cong"] & \KK^{\MW}_2 \\
\underline{H}_1^{\A^1}\!(X) \arrow[ur, dashed, twoheadrightarrow] \\ \KK^{\MW}_2 \arrow[u, "\cong"] \arrow[uurrr, twoheadrightarrow, out=0, in=205, "\varphi"]
\end{tikzcd}
\end{center}
where $\varphi \colon \KK^{\MW}_2 \twoheadrightarrow \KK^{\MW}_2$ is just the composition. It is known that the endomorphisms of $\KK^{\MW}_2$ are $K^{\MW}_0(k)$, where such an element acts by multiplication; see, for instance, Corollary 5.2 of \cite{MR4874195}. So the epimorphism $\psi \circ \varphi$ agrees with the multiplication of an element of $K^{\MW}_0(k)$. Since $\KK^{\MW}_2$ is a free strictly $\A^1$-invariant abelian sheaf, it is projective in $\Ab_{\A^1}(k)$, so this epimorphism is split. This allows us to choose an endomorphism $\psi \in \End_{\Ab_{\A^1} \!/ k}(\KK^{\MW}_2) \cong K^{\MW}_0(k)$ with $\id = \varphi \circ \psi$. As $K^{\MW}_0(k)$ is central in Milnor--Witt K-theory, we have $\id = \varphi \circ \psi = \psi \circ \varphi$ so that $\varphi$ is a isomorphism. In particular, the kernel of 
$$\underline{\pi}_1^{\A^1}\!(X) \twoheadrightarrow L_{\A^1}\bigl(\underline{\pi}_1^{\A^1}\!(X)/p_*\underline{\pi}_1^{\A^1}\!(\widetilde{X})\bigr) \cong \underline{\pi}_1^{\A^1}\!(X)/p_*\underline{\pi}_1^{\A^1}\!(\widetilde{X})$$
agrees with the kernel of of $h$ as we wanted to show.
\end{proof}
\begin{rem}
While we are currently not able to show that $\ker(h) = [\underline{\pi}_1^{\A^1}\!(X),\underline{\pi}_1^{\A^1}\!(X)]$,  we are able to show that $\ker(h)$ agrees with $L_{\A^1}[\underline{\pi}_1^{\A^1}\!(X),\underline{\pi}_1^{\A^1}\!(X)]$. Indeed, if we denote the maps 
$$\KK^{\MW}_2 \twoheadrightarrow L_{\A^1}\underline{\pi}_1^{\A^1}\!(X)/[\underline{\pi}_1^{\A^1}\!(X),\underline{\pi}_1^{\A^1}\!(X)] \,\,\, \text{and} \,\,\, L_{\A^1}\underline{\pi}_1^{\A^1}\!(X)/[\underline{\pi}_1^{\A^1}\!(X),\underline{\pi}_1^{\A^1}\!(X)] \twoheadrightarrow \KK^{\MW}_2$$
from the proof by $\alpha$ and $\beta$ respectively, then our argument from above shows that $\id = \psi \circ \varphi = \psi \circ \beta \circ \alpha$, so that $\alpha$ is a monomorphism. Since we are working in an abelian and hence in a balanced category, $\alpha$ is an isomorphism, so that also 
$$\underline{H}_1^{\A^1}\!(X) \to L_{\A^1}\bigl(\underline{\pi}_1^{\A^1}\!(X)/[\underline{\pi}_1^{\A^1}\!(X),\underline{\pi}_1^{\A^1}\!(X)]\bigr)$$
is an isomorphism. As a consequence, also the quotient map
$$L_{\A^1}\bigl(\underline{\pi}_1^{\A^1}\!(X)/[\underline{\pi}_1^{\A^1}\!(X),\underline{\pi}_1^{\A^1}\!(X)]\bigr) \twoheadrightarrow L_{\A^1}\bigl(\underline{\pi}_1^{\A^1}\!(X)/p_*\underline{\pi}_1^{\A^1}\!(\widetilde{X})\bigr)$$
is an isomorphism as a composition of isomorphisms. Using that $L_{\A^1}$ is a left adjoint and hence respects quotients, together with the strong $\A^1$-invariance of $\A^1$-fundamental sheaves, then gives 
$$L_{\A^1}[\underline{\pi}_1^{\A^1}\!(X),\underline{\pi}_1^{\A^1}\!(X)] = L_{\A^1}p_*\underline{\pi}_1^{\A^1}\!(\widetilde{X}) \cong L_{\A^1}\underline{\pi}_1^{\A^1}\!(\widetilde{X}) \cong \underline{\pi}_1^{\A^1}\!(\widetilde{X}).$$
By the above theorem, the sheaf on the very right is $\ker(h)$. In particular, to prove that the Hurewicz map for algebraic knot complements $X$ is the abelianization, it suffices to prove that the commutator subsheaf $[\underline{\pi}_1^{\A^1}\!(X),\underline{\pi}_1^{\A^1}\!(X)]$ is strongly $\A^1$-invariant.

By modifying the first part of this argument, we can also show that the second contraction $\ker(h)_{-2}$ agrees with $[\underline{\pi}_1^{\A^1}\!(X),\underline{\pi}_1^{\A^1}\!(X)]_{-2}$. To show this, we consider the quotient map $\KK^{\MW}_2 \twoheadrightarrow \KK^{\MM}_2$, then contract twice to obtain $(\KK^{\MM}_2)_{-2} \cong \KK^{\MM}_0 = \underline{\Z}$ and then define a map to $\bigl(\underline{\pi}_1^{\A^1}\!(X)/[\underline{\pi}_1^{\A^1}\!(X),\underline{\pi}_1^{\A^1}\!(X)]\bigr)_{-2}$ by using the freeness of $\underline{\Z}$. From there on we can basically argue as above.
\end{rem}
So we have our desired coverings and they satisfy the properties that fit the analogy to the classical theory. As can be done there, these allow us to modify our coefficients for (co-)homology, but the definition is a bit less clear. Given an abelian sheaf $M \colon \Sm_k \rightarrow \Ab$, we consider its left Kan extension along the functor $\Sm_k \rightarrow \Spc(k)$, which we also denote by $M$ and call an abelian sheaf on $\Spc(k)$. 
\begin{de}
Let $X$ be an algebraic knot complement, $p \colon \widetilde{X} \rightarrow X$ its universal $\KK^{\MW}_2$-covering and let $M$ be an abelian sheaf on $\Spc(k)$. We call the sheaf $p_*M = M(- \times_X \widetilde{X}) \in \Shv(X_{\Nis})$ the pushforward of $M$ along $p$.
\end{de}
Based on this we can now introduce Alexander modules of algebraic knots.
\begin{de}
Let $i \colon \A^1 \hookrightarrow \A^3$ be an algebraic knot, $X$ its knot complement, let $p \colon \widetilde{X} \rightarrow X$ be the universal $\KK^{\MW}_2$-covering and let $M$ be an abelian sheaf on $\Spc(k)$. The Alexander module of $i$ with coefficients in $M$ is the homology sheaf $\underline{H}_1^{\A^1}\!(\widetilde{X},M) \cong \underline{H}_1^{\A^1}\!(X,p_*M)$.
\end{de}
If $M = \underline{\Z}$, we just speak of the Alexander module of $i$ and denote it by $\underline{H}_1^{\A^1}\!(\widetilde{X})$.
\begin{rem}
We can also consider cohomological Alexander modules $H^1(\widetilde{X},M)$ by using sheaf cohomology. Here the underlying Grothendieck topology is relevant, as the Zariski and Nisnevich sheaf cohomology groups are known to agree when their coefficients are strictly $\A^1$-invariant; see Corollary 5.43 of \cite{MR2934577}. While this notion fits the analogy with the topological story less well than its homological counterpart, it might be more suitable for algebro-geometric arguments.
\end{rem}
\begin{prop}\label{Alexander mods detect alg knot types}
Alexander modules are invariants of algebraic knot types, i.e., if $i \colon \A^1 \hookrightarrow \A^3$ and $i' \colon \A^1 \hookrightarrow \A^3$ are algebraic knots for which there exists an automorphism $\varphi \in \Aut(\A^3)$ with $\varphi \circ i = i'$, then the Alexander modules of these algebraic knots agree.
\end{prop}
\begin{proof}
Denote the complements of $i$ and $i'$ by $X$ and $X'$, respectively. Since $\varphi$ is an automorphism that satisfies $\varphi \circ i = i'$, it restricts to an isomorphism
$$\varphi \colon X = \A^3 \setminus i(\A^1) \rightarrow \A^3 \setminus i'(\A^1) = X'.$$
By Proposition \ref{Homology of Knot Complement} we thus obtain an automorphism $\varphi_* \in \Aut(\KK^{\MW}_2)$ after applying $\underline{H}_1^{\A^1}(-)$, which yields a commutative diagram
\begin{center}
\begin{tikzcd}
\widetilde{X} \arrow[r, "\simeq"] \arrow[d, "p"] & \widetilde{X'} \arrow[d, "p'"] \\
X \arrow[r, "\simeq"] & X'.
\end{tikzcd}
\end{center}
Therefore we obtain an isomorphism $H_1^{\A^1}\!(\widetilde{X}) \rightarrow H_1^{\A^1}\!(\widetilde{X'})$.
\end{proof}
To produce an explicit example of an Alexander module $\underline{H}_1^{\A^1}\!(\widetilde{X},M) \cong \underline{H}_1^{\A^1}\!(X,p_*M)$, let us present one idea on how to approach this object. The coefficients $p_*M$ are not that easily dealt with. If we work on open subschemes $U_i$ of $X$ where the covering $p$ trivializes, then we can understand the sheaf $p_*M$ well enough. Its value is given by
\begin{align*}
p_*M(U_i) = M(U_i \times_X \widetilde{X}) = M(U_i \times \KK^{\MW}_2) \simeq \Map(U_i \times \KK^{\MW}_2, M) & \simeq \Map(U_i,\underline{\Map}(\KK^{\MW}_2, M)) \\ & \simeq \Map(U_i, \underline{\Hom}(\KK^{\MW}_2,M)) \\ & \simeq \underline{\Hom}(\KK^{\MW}_2,M)(U_i),
\end{align*}
where $\underline{\Hom}$ denotes the discrete mapping space. This does perhaps not seem like something we understand better, but by prior work of the second named author \cite{MR4874195}, we actually do. Let us state the result without explaining some of the notions for now, and then discuss those a bit afterwards.
\begin{thm}\label{Homs out of KMW}
If $M_*$ is a connective homotopy module with ring structure, then 
$$\Hom_{\Shv(\Sm_k)}(\KK^{\MW}_2,M_*) \cong M_*(k)^{\N}$$
as filtered $M_*(k)$-modules. In particular, we have
$$\Hom_{\Shv(\Sm_k)}(\KK^{\MW}_2,M_n) \cong \prod_{l \geq 0} M_{n-2l}(k)$$
for all $n \geq 0$.
\end{thm}
So for suitable coefficients, namely connective homotopy modules with ring structure and their homogeneous pieces, we do understand the morphisms out of $\KK^{\MW}_2$. Homotopy modules are a motivic analogue of abelian groups from classical stable homotopy theory. By this we mean that there is a natural $t$-structure on the category of motivic spectra $\SH(k)$ determined by the vanishing of homotopy sheaves, such that 
$$\underline{\pi}_0(-)_* = \underline{\pi}_{-*,-*} \colon \SH(k)^{\heartsuit} \rightarrow \Pi(k)$$
is an equivalence, where the latter category is the category of homotopy modules. For readers who do not find this helpful, lets us give two central examples which also satisfy the additional assumptions of the above theorem.
\begin{ex}
In Definition \ref{Def of KMW}, we saw the Milnor--Witt K-theory sheaf $\KK^{\MW}_*$. This is a homotopy module with ring structure, but it is non-connective. Once we kill the unique generator $\eta$ of negative degree, i.e., we consider its quotient $\KK^{\MM}_* = \KK^{\MW}_* \! \! / \eta$, it becomes a connective homotopy module with ring structure, known as Milnor K-theory. Another such example is given by algebraic K-theory $\KK^{\QQ}_*$.
\end{ex}
So if we for instance consider the unknot $u \colon \A^1 \hookrightarrow \A^3$, we have an explicit cover of the knot complement on which the universal $\KK^{\MW}_2$-covering trivializes, as seen in Construction \ref{Construction of Milnor--Witt K_2 coverings}. Using the Mayer-Vietoris sequence together with Theorem \ref{Homs out of KMW}, we can compute the homology of $U_i$ and also control how these groups glue. While this is one way to compute the Alexander module of the unknot, we also have a rather formal argument based on our previous results. Let us present that one instead.
\begin{ex}\label{Alexander mod of unknot}
The Alexander module of the unknot is trivial. Indeed, we have seen in Lemma \ref{Unknot is Univ Covering} that the $\KK^{\MW}_2$-covering $\widetilde{X} \rightarrow X$ of the complement of the unknot is universal, that is, we have $\underline{\pi}_1^{\A^1}\!(\widetilde{X}) = 0$. The weak Hurewicz theorem (Theorem 6.35 of \cite{MR2934577}) now tells us, that for any morphism $\alpha \colon \underline{\pi}_1^{\A^1}\!(\widetilde{X}) \rightarrow M$ into a strictly $\A^1$-invariant abelian sheaf $M$, there is a unique morphism of strictly $\A^1$-invariant abelian sheaves $\beta \colon \underline{H}_1^{\A^1}\!(\widetilde{X}) \rightarrow M$, such that $\alpha = \beta \circ h$. The condition on the composition is void since $\underline{\pi}_1^{\A^1}\!(\widetilde{X}) = 0$, so that $\underline{H}_1^{\A^1}\!(\widetilde{X})$ is just the initial strictly $\A^1$-invariant abelian sheaf, i.e., we have $\underline{H}_1^{\A^1}\!(\widetilde{X}) = 0$ as claimed.
\end{ex}
\begin{cor}\label{Rectifiability invariant}
If there exists an algebraic knot $i \colon \A^1 \hookrightarrow \A^3$ with knot complement $X$, such that the Alexander module $\underline{H}_1^{\A^1}\!(\widetilde{X})$ is non-trivial, then $i$ is not rectifiable.
\end{cor}
\begin{proof}
This follows directly from the above example together with Proposition \ref{Alexander mods detect alg knot types}.
\end{proof}
That the Alexander module of the unknot is trivial is maybe not too surprising, at least if the analogy to the classical story is strong enough. So what about other examples? The same approach as for the unknot works in theory, but it seems very difficult to execute. The first issue is that our construction of the universal $\KK^{\MW}_2$-coverings relies on the fact from Theorem \ref{Knots are Complete Intersections} that algebraic knots are complete intersections. At least for the algebraic trefoil knot this is not an obstruction anymore; see the discussion before Question \ref{Rectifiability}. To show that it defines a non-trivial algebraic knots, let's say, over $\C$, i.e., a non-rectifiable embedding $\A^1 \hookrightarrow \A^3$, we do not need complete computations of its Alexander module. Producing one non-trivial homology class suffices. This is, of course, still much easier said than done, but this is what we are currently working on \cite{WW2}.
\begin{rem}
Since our work is strongly inspired by classical knot theory, a natural question is if we can recover those classical notions. Yes, we can! There is real realization functor $r_{\R} \colon \Spc(\R) \rightarrow \Spc$ induced by the $\R$-valued points functor $(-)(\R)$. Under this functor, the discrete motivic space $\KK^{\MW}_2$ is mapped to $\Z$, the universal $\KK^{\MW}_2$-coverings are mapped to universal $\Z$-coverings and the motivic Alexander module $\underline{H}_1^{\A^1}\!(\widetilde{X})$ is mapped to the Alexander module of the realization of $X$. This does in particular show that those geometric notions can be defined algebraically. Also for this we refer to a forthcoming article, which may or may not differ from \cite{WW2}.
\end{rem}

\appendix
\section{Algebraic Knots defined by Chebyshev Polynomials}\label{Appendix A}

For this entire section we consider a base field $k$ of characteristic $0$. We recall some facts concerning algebraic knots defined by Chebyshev polynomials from \cite{freudenburg:freudenburg}. This provides infinitely many algebraic knot embeddings, and the key point about zeros of Chebyshev polynomials being real and simple will be relevant for our discussion of the surjectivity of the Hurewicz map for knot complements. 

First, recall that Chebyshev polynomials of the first kind are polynomials $T_n$ defined by $T_n(\cos\theta)=\cos(n\theta)$, and that these polynomials satisfy a recursion of the form $T_0(x)=1, T_1(x)=x$ and 
\[
T_{n+1}(x)=2xT_n(x)-T_{n-1}(x)
\]
The Chebyshev polynomials of the first kind give rise to a family of monic polynomials by
\[
F_n(x)=2T_n\left(\frac{x}{2}\right)
\]

and the polynomials $T_n$ (or $F_n$) can be used to define embeddings 
\[
\iota(i,j,k)\colon\A^1\hookrightarrow\A^3\colon t\mapsto (T_i(t),T_j(t),T_k(t)).
\]

\begin{ex}
We list a couple of interesting example cases. More examples can be found in \cite[Table 2]{freudenburg:freudenburg}. Whenever we mention certain knots, we implicitly assume that the base field contains $\R$.
\begin{itemize}
    \item The embeddings $\iota(2,j,k)$ and $\iota(i+j,i+j)$ are rectifiable, by \cite[Corollaries 4.1 and 4.2]{freudenburg:freudenburg}.
    \item The embedding $\iota(3,4,5)$ is the trefoil knot. Explicitly, (the monic version of) it is defined by 
    \[
    t\mapsto (t^3-3t,t^4-4t^2+2,t^5-5t^3+5t).
    \]
    A coordinate on the embedded knot is given by $-x^3+yz+3x$.
    \item The embedding $\iota(3,5,7)$ is the figure-eight knot. Explicitly, (the non-monic version of) it is defined by 
    \[
    t\mapsto (4t^3-3t,16t^5-20t^3+5t,64t^7-112t^5+56t^3-7t). 
    \]
    A coordinate on the embedded knot is given by $-4xy^2+4x^2z+2x-z$.
    \item By results of Koseleff and Pecker \cite{koseleff:pecker}, the embeddings $\iota(3,3n+1,3n+2)$ represent the $(2,2n+1)$ torus knots, and the embeddings $\iota(i,j,ij-i-j)$ for $(i,j)=1$ represent the alternating $(i,j)$-knots. In particular, the Chebyshev embeddings $\iota(i,j,k)\colon\A^1\to\A^3$ represent infinitely many knot types.
    \item Another result of Koseleff and Pecker \cite{koseleff:pecker} shows that if we allow an additional phase, i.e., define the embedding as $t\mapsto (T_i(t),T_j(t),T_k(t+\phi))$, then every knot has a Chebyshev polynomial representation.
\end{itemize} 
\end{ex}

We close this section with a well-known fact about the roots of Chebyshev polynomials, which can, for instance, be found as lemma 2.3 (b) in \cite{freudenburg:freudenburg}:
\begin{prop}
    The roots of the Chebyshev polynomial $T_n \in \R[t]$ are $\cos\left(\frac{2k-1}{2n}\pi\right)$ for all $1 \leq k \leq n$, and all these roots are simple. 
\end{prop}

\begin{cor}
\label{cor:chebyshev-simple}
    Let $\iota(i,j,k)\colon\A^1\to\A^3$ be an embedding defined either by the Chebyshev polynomials $T_n$ or their monic counterparts $F_n$ over the field $\R$. Then the intersection points of the image $\iota(i,j,k)(\A^1)$ with any coordinate hyperplane in $\A^3$ are all real and simple.
\end{cor}

\section{Notes on \texorpdfstring{$\A^1$}{A1}-Fundamental Groups of Knot Complements}\label{Appendix B}

In this appendix we discuss a way of proving that the Hurewicz homomorphism $\underline{\pi}_1^{\A^1}\!(X)\to \underline{H}_1^{\A^1}\!(X)$ is an epimorphism for knot complements, assuming a special condition satisfied if $k = k^{\alg}$ or by algebraic knots defined by Chebyshev polynomials. The main idea of the proof is as follows. For an algebraic knot $i \colon \A^1 \to \A^3$, consider a hyperplane $H\cong\A^2$ in $\A^3$ such that all intersection points in $H \cap i(\A^1)$ are $k$-rational and simple. For Chebyshev knots, these exist and can be chosen to be coordinate hyperplanes, by Corollary~\ref{cor:chebyshev-simple}. For $H \setminus (H \cap i(\A^1))$, an excision result of Asok and Doran implies surjectivity of the Hurewicz map. Via an $\A^1$-local degree argument, the transversality of the intersection points implies surjectivity of the composition
\[
\underline{\pi}_1^{\A^1}\!(H\setminus(H\cap\iota(\A^1)))\to\underline{\pi}_1^{\A^1}\!(X)\to \underline{H}_1^{\A^1}\!(X)\cong\K^{\MW}_2.
\]
This will then establish the surjectivity of the Hurewicz map for knot complements satisfying the hyperplane condition from above.

We begin with a consequence of the Asok--Doran excision theorem for $\A^1$-homotopy groups:

\begin{prop}\label{prop:hurewicz-plane-minus-pts}
   Let $k$ be an infinite base field of characteristic not $2$. Then the Hurewicz map 
   $$h \colon \underline{\pi}_1^{\A^1}\!(\A^2 \setminus \{ p_1,\dots,p_n \})\to \underline{H}_1^{\A^1}\!(\A^2 \setminus \{ p_1,\dots,p_n \})$$
   is an epimorphism, where  $p_1,\dotsc,p_n$ are closed $k$-rational points of $\A^2$.
\end{prop}

\begin{proof}
We split the argument into three steps.

\underline {Step 1:} Since $p_1$ is rational, we may assume that $p_1=0$ is the origin, after applying a suitable translation. Now we consider the inclusion $j \colon \A^2 \setminus \{ 0,p_2,\dots,p_n \} \hookrightarrow \A^2 \setminus \{ 0 \}$ obtained by removing the remaining closed points $p_2,\dots,p_n$. This open subscheme is $\A^1$-connected, which follows, for instance, from \cite[Remark~2.16]{asok:doran:excision}. Therefore, the excision theorem of Asok and Doran \cite[Theorem~4.1]{asok:doran:excision} for $d=2$ and $m=0$ yields that the map
\[
j_*\colon \underline{\pi}_i^{\A^1}(\A^2 \setminus \{ 0,p_2,\dots,p_n \}) \to \underline{\pi}_i^{\A^1}(\A^2 \setminus \{ 0 \})
\]
is an epimorphism for $i=1$. Here we note that the notation in \cite{asok:doran:excision} differs from ours and their excision theorem is indeed about homotopy sheaves. This implies that the composition 
\[
\underline{\pi}_1^{\A^1}(\A^2 \setminus \{ 0,p_2,\dots,p_n \}) \to \underline{\pi}_1^{\A^1}(\A^2 \setminus \{ 0 \}) \to \underline{H}_1^{\A^1}(\A^2 \setminus \{ 0 \})\cong\K^{\MW}_2
\]
with the isomorphism $h \colon \underline{\pi}_1^{\A^1}(\A^2 \setminus \{ 0 \})\to\underline{H}_1^{\A^1}(\A^2 \setminus \{ 0 \})$ from the end of Example \ref{Homotopy Sheaf Example} is an epimorphism.

\underline {Step 2:} Note that since $\Th(\mathcal{N}_{\{0,p_2,\dotsc,p_n\} /\A^2}) \simeq \bigvee_{i = 1}^n \Th(\mathcal{N}_{\{p_i\} /\A^2})$, the proof of Proposition \ref{Homology of Knot Complement} yields an isomorphism 
$$\underline{H}_1^{\A^1}\!(\A^2 \setminus \{ 0,\dots,p_n \}) \cong \bigoplus_{i = 1}^n \underline{H}_2^{\A^1}\!( \Th(\mathcal{N}_{\{p_i\} /\A^2})) \cong (\underline{H}_2^{\A^1}\!(S^{4,2}))^{\oplus n} \cong (\KK^{\MW}_2)^{\oplus n},$$
one copy of $\KK^{\MW}_2$ for each of the $k$-rational points $p_i$. Now we consider the diagram
\begin{center}
\begin{tikzcd}
\pi_1^{\A^1}\!(\A^2\setminus\{0,p_2,\dots,p_n\}) \arrow[r, twoheadrightarrow, "j_*"] \arrow[d,"h"] & \pi_1^{\A^1}\!(\A^2\setminus\{0\}) \arrow[d,"\cong"] \\
H_1^{\A^1}\!(\A^2\setminus\{0,p_2,\dots,p_n\}) \arrow[r,"j_*"] \arrow[d,"\cong"] & H_1^{\A^1}\!(\A^2\setminus\{0 \} \arrow[d,"\cong"] \\
(\KK^{\MW}_2)^{\oplus n} \arrow[r, twoheadrightarrow, "\pr_1"] & \KK^{\MW}_2
\end{tikzcd}
\end{center}
Here the vertical maps are the Hurewicz maps and the horizontal ones are induced by the inclusion $j \colon \A^2 \setminus \{ 0,p_2,\dots,p_n \} \hookrightarrow \A^2 \setminus \{ 0 \}$, so that this diagram commutes. By step 1, the top horizontal map is an epimorphism and hence also the composition 
\begin{center}
\begin{tikzcd}
\pi_1^{\A^1}\!(\A^2 \setminus \{ 0,p_2,\dots,p_n \}) \arrow[r, "h"] & H_1^{\A^1}\!(\A^2 \setminus \{ 0,p_2,\dots,p_n \}) \arrow[r, "\pr_1"] & \KK^{\MW}_2
\end{tikzcd}
\end{center}
of the Hurewicz map followed by the projection to the $\KK^{\MW}_2$-summand corresponding to the point $0\in \A^2$.

\underline {Step 3:} The arguments in steps 1 and 2 work for any of the points $p_i$, since we assumed all of them to be $k$-rational. This means that the composition of the Hurewicz map
\[
h \colon \underline{\pi}_1^{\A^1}\!(\A^2 \setminus \{ p_1,\dots,p_n \}) \to \underline{H}_1^{\A^1}\!(\A^2 \setminus \{ p_1,\dots,p_n \}) \cong (\KK^{\MW}_2)^{\oplus n}
\]
with any of the $n$ projections $\pr_i\colon(\KK^{\MW}_2)^{\oplus n}\to\KK^{\MW}_2$ is an epimorphism. But this means that the Hurewicz map is itself an epimorphism, as claimed.
\end{proof}

\begin{rem}
By analogy to the topological version, one would expect the $\A^1$-fundamental sheaf of $\A^2 \setminus \{ p_1,\dots,p_n \}$ to be a free product of $n$ copies of $\KK^{\MW}_2$. A formula of Asok--Morel \cite{asok:morel} says something in this direction, showing an equivalence $\A^2 \setminus \{ p_1,\dots,p_m \} \simeq \Sigma \! \Gm \wedge \A^1 \setminus \{ p_1,\dots,p_m \}$ of motivic spaces. This yields
\[
\underline{\pi}_1^{\A^1}\!(\A^2\setminus\{p_1,\dots,p_m\})\cong \underline{H}_0^{\A^1}\!(\Gm \wedge \A^1 \setminus\{ p_1,\dots,p_m \}),
\]
which still seems some way from a free product formula, and also doesn't seem to immediately imply that the Hurewicz map is an epimorphism. 
\end{rem}

Now that we established that the Hurewicz map for $\A^2 \setminus \{ p_1,\dots,p_m \}$ is an epimorphism, we need some statements about $\A^1$-homology. Essentially, given an algebraic knot $i \colon \A^1 \hookrightarrow \A^3$ with complement $X = \A^3 \setminus i(\A^1)$ and a hyperplane $\A^2 \cong H \subset \A^3$, we consider the composition 
\[
\A^2\setminus\{p_1,\dots,p_n\}\cong H\setminus(H\cap i(\A^1))\subset X\xrightarrow{(f,g)}\A^2\setminus\{0\}
\]
where the map $(f,g)$ is defined as in Construction \ref{Construction of Milnor--Witt K_2 coverings}. As above, the induced map on $\A^1$-homology is of the form 
$$(\K^{\MW}_2)^{\oplus n} \cong \underline{H}_1^{\A^1} \! (\A^2 \setminus \{ p_1,\dots,p_n \}) \to \underline{H}_1^{\A^1} \! (\A^2 \setminus \{ 0 \}) \cong \KK^{\MW}_2,$$
and is given by the sum of the local $\A^1$-degrees of the above composite map $\A^2\setminus\{p_1,\dots,p_n\}\to\A^2\setminus\{0\}$ at the points $p_i$. Indeed, the $i$-th summand is a morphism
$$\KK^{\MW}_2 \cong \underline{H}_2^{\A^1}\!(S^{4,2}) \cong \underline{H}_2^{\A^1}\!(\Th(\mathcal{N}_{\{p_i \}/\A^2})) \to \underline{H}_2^{\A^1}\!(\Th(\mathcal{N}_{\{0 \}/\A^2})) \cong \underline{H}_2^{\A^1}\!(S^{4,2}) \cong \KK^{\MW}_2.$$
But all such morphisms are known to be given by multiplication with an element of $\K^{\MW}_0(k) = \GW(k)$, as can be found in \cite[Corollary 5.2]{MR4874195}. Via homotopy purity, this element is by definition the local $\A^1$-degree of $\A^2\setminus\{p_1,\dots,p_n\}\to\A^2\setminus\{0\}$ at the point $p_i$.
\begin{prop}
\label{prop:a1-local-degree}
In the above situation, the local $\A^1$-degree of the composite map 
\[
\A^2\setminus\{p_1,\dots,p_n\}\cong H\setminus(H\cap i(\A^1))\subset X\xrightarrow{(f,g)}\A^2\setminus\{0\}. 
\]
at any of the points $p_i$ is invertible. 
\end{prop}

\begin{proof}
By a theorem of Kass and Wickelgren \cite{kass:wickelgren:local-degree}, the local $\A^1$-degree as an element in the abelian group $\GW(k) \cong [(\mathbb{P}^1)^{\wedge 2},(\mathbb{P}^1)^{\wedge 2}]$ can be computed as the Eisenbud--Khimshiashvili--Levine class. 

We collect some algebraic information to compute that class. A key point is that according to Proposition \ref{Knots are Complete Intersections}, the knot $i(\A^1)$ is a complete intersection, defined by the functions we denoted $(f,g)$ throughout. We also have a coordinate function $\varphi \colon \A^3 \to \A^1$, whose composition with the embedding $i\colon\A^1\to\A^3$ is the identity, showing that $i$ is an embedding. Consequently, at each point $p$ of the knot, $f$ and $g$ generate the conormal bundle of the knot, and $\varphi$ generates its tangent bundle. If we consider a hyperplane $\A^2 \cong H \subset \A^3$ such that all intersection points are simple and $k$-rational, then the restrictions of $f$ and $g$ to the hyperplane $\A^2 \cong H$ generate the tangent space to $H$ at any of the intersection points $p_i$. 

Now we can compute the EKL-form. At any of the ($k$-rational) intersection points $p$, the restrictions of $f$ and $g$ generate the maximal ideal of the point $\mathfrak{m}_p$. Hence the local $k$-algebra $Q=k[X,Y]_p/(f,g)$ is isomorphic to $k$. Now we claim that the Jacobian $J$ of the composite $\A^2\cong H\hookrightarrow \A^3\to\A^2$ of $(f,g)$ and the inclusion is invertible at $p$. One way to see this is to use adapted coordinates. By definition, $V(f)$ and $V(g)$ map to the coordinate axes of $\A^2$ under the above composition and we can therefore choose the tangents to $V(f)$ and $V(g)$ as basis vector for the tangent space. With this basis, the Jacobian is the identity. Thus, for any other choice of basis, it is invertible. The EKL-form is now the quadratic form 
\[
\omega^{\rm EKL}\colon Q\times Q\to k\colon(v,w)\mapsto \frac{vw}{\det(J)}
\]
As an element in $\GW(k)$, we can write it as $\langle \det J^{-1}\rangle$ and see that it is invertible as claimed.
\end{proof}

\begin{thm}
\label{thm:hurewicz-knot-complement}
    Let $k$ be an infinite base field of characteristic not $2$ and let $i\colon \A^1\hookrightarrow\A^3$ be an algebraic knot with complement $X = \A^3 \setminus i(\A^1)$. Suppose there exists a hyperplane $\A^2  \cong H \subset \A^3$ that intersects $i(\A^1)$ transversely and all intersection points in $H\cap i(\A^1)$ are $k$-rational. Then the Hurewicz map $h\colon\underline{\pi}_1^{\A^1}\!(X)\to\underline{H}_1^{\A^1}\!(X)$ is an epimorphism. 
\end{thm}

\begin{proof}
By assumption, $\A^2\setminus\{p_1,\dots,p_n\} \cong H \setminus (H\cap i(\A^1))= H \cap X$ embeds into $X$. Then we consider the composition
\[
\A^2\setminus\{p_1,\dots,p_n\}\hookrightarrow X \xrightarrow{(f,g)} \A^2\setminus\{0\},
\]
where $(f,g)$ is defined as in Construction \ref{Construction of Milnor--Witt K_2 coverings}. This gives rise to the following commutative diagram:
\begin{center}
\begin{tikzcd}
\underline{\pi}_1^{\A^1}\!(\A^2\setminus\{p_1,\dots,p_n\}) \arrow[r] \arrow[d, "h"] & \underline{\pi}_1^{\A^1}\!(X) \arrow[d,"h"] \arrow[r] & \underline{\pi}_1^{\A^1}\!(\A^2\setminus\{0\}) \ar[d,"\cong"]\\
\underline{H}_1^{\A^1}\!(\A^2\setminus\{p_1,\dots,p_n\}) \arrow[r] & \underline{H}_1^{\A^1}\!(X) \arrow[r,"\cong"] & \KK^{\MW}_2
\end{tikzcd}
\end{center}
The upper row of the diagram is obtained by applying $\underline{\pi}_1^{\A^1}$ to the composition above and the lower one by applying $\underline{H}_1^{\A^1}$. All the vertical arrows are Hurewicz maps, and the commutativity follows from the naturality of Hurewicz maps. To show that the middle vertical map is an epimorphism, it suffices to show that the composite map across the whole rectangle is an epimorphism. Indeed, then the bottom composite of the left square is an epimorphism since the bottom horizontal map in the right square is an isomorphism, and by commutativity this implies that $h\colon\underline{\pi}_1^{\A^1}(X)\to\underline{H}_1^{\A^1}(X)$ is an epimorphism.

For showing that the composite map across the whole rectangle is an epimorphism, we argue as follows. The composition 
\[
\underline{H}_1^{\A^1}\!(\A^2\setminus\{p_1,\dots,p_n\})\to \underline{H}_1^{\A^1}\!(X)\to \underline{H}_1^{\A^1}\!(\A^2\setminus\{0\})\cong\KK^{\MW}_2
\]
in the bottom row of the above diagram is the sum of the maps
\begin{center}
\begin{tikzcd}[row sep = large, column sep = 12ex]
\underline{H}_1^{\A^1}\!(\A^2\setminus\{p_1,\dots,p_n\}) \cong (\KK^{\MW}_2)^{\oplus n} \arrow[r,"\pr_i"] & \KK^{\MW}_2 \arrow[r, "\deg^{\A^1}\!((f\text{,}\hspace{1pt}g)_i)"] & \underline{H}_1^{\A^1}\!(\A^2\setminus\{0\}) \cong \KK^{\MW}_2
\end{tikzcd}
\end{center}
of the projection to the $i$-th summand with multiplication by the $\A^1$-local degree of the composition
\begin{center}
\begin{tikzcd}
\A^2\setminus\{p_1,\dots,p_n\} \arrow[r] & X \arrow[r, "(f\text{,}\hspace{1pt}g)"] & \A^2 \setminus \{ 0 \}
\end{tikzcd}
\end{center}
at the point $p_i$. By Proposition~\ref{prop:a1-local-degree}, these degrees are invertible and hence in particular epimorphisms. By Proposition~\ref{prop:hurewicz-plane-minus-pts}, the left vertical Hurewicz map is an epimorphism, so that the considered map is a composite of two epimorphisms. This shows the claim.
\end{proof}

\begin{rem}\label{rem:assumptions for hurewicz}
The hyperplane condition of Proposition~\ref{thm:hurewicz-knot-complement} is satisfied  for any algebraic knot over $k=k^{\alg}$, and it is satisfied for Chebyshev knots over $\mathbb{R}$, by Corollary \ref{cor:chebyshev-simple}. In fact, for each Chebyshev knot there exists a number field over which the condition is satisfied.
\end{rem}
\begin{rem}
The basic idea of the above proof is inspired from topology where we can simply draw a simple loop around the knot which provides a section of the Hurewicz map $\pi_1(X)\to H_1(X)\cong\Z$. In the algebraic setting, this would require a map $\A^2\setminus\{0\}\to X$ whose composition with $(f,g)$ is (homotopic to) the identity, but we have no idea how to find such a map. The next best idea is to consider a combination of several loops in the above hyperplane section. The topological picture would also suggest that the map $\underline{\pi}_1^{\A^1}\!(\A^2\setminus\{p_1,\dots,p_n\})\to \underline{\pi}_1^{\A^1}\!(X)$ induced by the hyperplane section (in a sort of Wirtinger-style presentation) is an epimorphism. Again, we have no idea how one could prove this; it is also not entirely clear if a description of $\underline{\pi}_1^{\A^1}\!(X)$ of this form should be expected.
\end{rem}

\bibliographystyle{siam}
\bibliography{references}

\end{document}